\documentclass[a4paper,11pt]{amsart}
\usepackage{amssymb, amsmath}
\usepackage{amscd}
\numberwithin{equation}{section}
\usepackage{epsfig}
\usepackage{amsmath}
\usepackage{amsfonts,amssymb,amsopn}

\usepackage{amsthm}
\usepackage{hyperref}
\usepackage{verbatim}
\usepackage{color}
\usepackage[color,all]{xy}

\newcommand{\cA}{{\mathcal A}}
\newcommand{\cB}{{\mathcal B}}
\newcommand{\cF}{{\mathcal F}}

\newcommand{\cL}{{\mathcal L}}
\newcommand{\cM}{{\mathcal M}}
\newcommand{\cO}{{\mathcal O}}
\newcommand{\cS}{{\mathcal S}}
\newcommand{\cU}{{\mathcal U}}

\newcommand{\cW}{{\mathcal W}}
\newcommand{\C}{{\mathbb C}}
\newcommand{\K}{{\mathbb K}}
\newcommand{\bL}{{\mathbb L}}
\newcommand{\PP}{{\mathbb P}}
\newcommand{\bW}{{\mathbb W}}
\newcommand{\bZ}{{\mathbb Z}}

\newcommand{\tY}{\tilde{Y}}
\newcommand{\tM}{\widetilde{\cM}}

\newcommand{\Sym}{\mathrm{Sym}}
\newcommand{\End}{\mathrm{End}\,}
\newcommand{\Ker}{\mathrm{Ker}}
\newcommand{\Aut}{\mathrm{Aut}\,}
\newcommand{\im}{\mathrm{Im}\,}
\newcommand{\Iden}{\mathrm{Id}}
\newcommand{\Hom}{\mathrm{Hom}}
\newcommand{\rank}{\mathrm{rk}\,}
\newcommand{\Pic}{\mathrm{Pic}}
\newcommand{\ev}{\mathrm{ev}}
\newcommand{\sym}{\mathrm{sym}}
\newcommand{\isom}{\xrightarrow{\sim}}

\newcommand{\Oc}{{{\cO}_C}}

\newcommand{\Sec}{\mathrm{Sec}}
\newcommand{\Supp}{\mathrm{Supp}}
\newcommand{\Sing}{\mathrm{Sing}}

\newcommand{\Sp}{\mathrm{Sp}}
\newcommand{\Gr}{\mathrm{Gr}}

\newcommand{\Res}{\mathrm{Res}}

\newcommand{\Sn}{\cS_{2n, K}}
\newcommand{\Snk}{\Sn^k}
\newcommand{\bomega}{\overline{\omega}}
\newcommand{\bphi}{\bar{\phi}}
\newcommand{\be}{\bar{e}}

\newcommand{\bE}{\bar{E}}
\newcommand{\bI}{\bar{I}}
\newcommand{\bJ}{\bar{J}}

\newcommand{\Urd}{\cU (r, d)}
\newcommand{\Utn}{\cU (2n, 2n(g-1))}

\newcommand{\MS}{{\cM}{\cS} ( 2n, K )}
\newcommand{\SG}{SG_{2n, K}}

\newcommand{\SGkW}{SG^k ( \cW )}
\newcommand{\SkW}{\cS^k ( \cW )}
\newcommand{\xij}{x_i x_j}
\newcommand{\bp}{\bar{p}}
\newcommand{\mus}{\mu^{\mathrm{s}}}
\newcommand{\musL}{\mu^{\mathrm{s}}_\Lambda}
\newcommand{\psis}{\psi_{\mathrm{s}}}
\newcommand{\GaL}{G( \alpha; r, L , k )}
\newcommand{\GaK}{G( \alpha; 2n , K^n , k )}

\newtheorem{theorem}{{\textbf Theorem}}[section]
\newtheorem{proposition}[theorem]{{\textbf Proposition}}
\newtheorem{corollary}[theorem]{{\textbf Corollary}}
\newtheorem{lemma}[theorem]{{\textbf Lemma}}
\newtheorem{criterion}[theorem]{{\textbf Criterion}}
\newtheorem{defit}[theorem]{{\textbf Definition}}
\newtheorem{remit}[theorem]{{\textbf Remark}}

\newenvironment{remark}{\begin{remit}\rm}{\end{remit}}
\newenvironment{definition}{\begin{defit}\rm}{\end{defit}}
\newcounter{tmp}

\title{Brill--Noether loci on moduli spaces of symplectic bundles over curves}

\author{Ali Bajravani}
\address[A.\ Bajravani]{Department of Mathematics, Faculty of Basic Sciences, Azarbaijan Shahid Madani University, Tabriz, I.\ R.\ Iran., P.\ O.\ Box: 53751-71379.}
\email{bajravani@azaruniv.ac.ir}

\author{George H.\ Hitching}
\address[G.\ H.\ Hitching]{Oslo Metropolitan University, Postboks 4, St.\ Olavs plass, 0130 Oslo, Norway.}
\email{gehahi@oslomet.no}


\begin{document}

\begin{abstract} The symplectic Brill--Noether locus $\Snk$ associated to a curve $C$ parametrises stable rank $2n$ bundles over $C$ with at least $k$ sections and which carry a nondegenerate skewsymmetric bilinear form with values in the canonical bundle. This is a symmetric determinantal variety whose tangent spaces are defined by a symmetrised Petri map. We obtain upper bounds on the dimensions of various components of $\Snk$. We show the nonemptiness of several $\Snk$, and in most of these cases also the existence of a component which is generically smooth and of the expected dimension. As an application, for certain values of $n$ and $k$ we exhibit components of excess dimension of the standard Brill--Noether locus $B^k_{2n, 2n(g-1)}$ over any curve of genus $g \ge 122$. We obtain similar results for moduli spaces of coherent systems. \end{abstract}

\maketitle

\section{Introduction}

Let $C$ be a projective smooth curve of genus $g \ge 2$ and $\Urd$ the moduli space of stable vector bundles of rank $r$ and degree $d$ over $C$. A fundamental attribute of $\Urd$ is the stratification by \textsl{generalised Brill--Noether loci}
\[ B^k_{r, d} \ := \ \{ W \in \Urd : h^0 ( C, W ) \ge k \} . \]
This is a determinantal variety of expected dimension
\[ \beta^k_{r,d} \ := \ r^2 ( g - 1 ) + 1 - k ( k - d + r(g-1) ) . \]
Moreover, $B^{k+1}_{r, d} \subseteq \Sing(B^k_{r,d})$. If $r = 1$, one obtains the classical Brill--Noether loci on $\Pic^d (C)$, which are traditionally denoted $W^{k-1}_d (C)$. For a generic curve, the $B^k_{1, d}$ behave as regularly as possible: $B^k_{1, d}$ is nonempty of dimension $\beta^k_{1, d}$ if and only if $\beta^k_{1,d} \ge 0$, and furthermore irreducible if this dimension is positive; and $\Sing ( B^k_{1, d} ) = B^{k+1}_{1, d}$. See \cite{ACGH} for a full account of this story.

For $r \ge 2$, the situation is more complicated, even for a general curve. In recent years, much attention has been given to determining the components of $B^k_{r, d}$ for $r \ge 2$, together with their dimensions and singular loci. See \cite{GTiB} for a survey. Several generalisations have been studied, including coherent systems (see for example \cite{BGPMN} and \cite{New11}), generalised theta divisors (see \cite{Bea06} for an overview) and more generally twisted Brill--Noether loci (see \cite{TiB14} and \cite{HHN}).

A variant of $B^k_{r, d}$ which is of particular relevance for the present work is the \textsl{fixed determinant Brill--Noether locus}
\[ B^k_{r, L} \ := \ \{ W \in \Urd : h^0 ( C, W ) \ge k \hbox{ and } \det ( W ) = L \} \]
where $L$ is a fixed line bundle of degree $d$. Denote by $K$ the canonical bundle $T^*_C$. The locus $B^k_{2, K}$ has been studied extensively in \cite{BF}, \cite{Muk92}, \cite{Muk97}, \cite{TiB04}, \cite{TiB07}, \cite{LNP} and \cite{Baj19} (see also \cite{GN}), and we shall return to this below. The loci $B^k_{r, L}$ for other values of $r$ and $L$ are studied in \cite{Oss}, \cite{Oss-2}, \cite{LNS}, \cite{Zha} and elsewhere.

In the present work, we consider a different generalisation of $B^k_{2, K}$ to higher rank. For any bundle $W$ of rank two, there is a natural skewsymmetric isomorphism $W \isom W^* \otimes \det (W)$. In general, recall that a vector bundle $W$ is said to be \textsl{$L$-valued symplectic} if there is a skewsymmetric isomorphism $W \isom W^* \otimes L$ for some line bundle $L$; equivalently, if there is a nondegenerate skewsymmetric bilinear form $\omega \colon W \otimes W \to L$. By nondegeneracy, a symplectic bundle must have even rank $2n \ge 2$, and moreover $\det(W) = L^n$. For us, $L$ will always be $K$. There is a quasiprojective moduli space $\cM\cS ( 2n, K )$ for stable $K$-valued symplectic bundles over $C$, which we discuss in more detail in {\S} \ref{background}. Our fundamental objects of study are the \textsl{symplectic Brill--Noether loci}
\[ \Snk \ := \ \{ W \in \MS : h^0 ( C, W ) \ge k \} \ \subseteq \ B^k_{2n, K^n} . \]

It follows from \cite[Remark 4.6]{Muk92} that $\Snk$ is a \emph{symmetric} determinantal variety of expected codimension $\frac{1}{2}k(k+1)$. In {\S} \ref{SchemeStructure}, we expand upon this remark, showing that $\Snk$ is \'etale locally defined by the vanishing of the $(k+1) \times (k+1)$-minors of a symmetric matrix. 
 In {\S} \ref{TangentSpaces} we recall a description of the Zariski tangent spaces of $\Snk$ in terms of a symmetrised Petri map. Adapting well-known results from \cite{ACGH} to the symplectic case, in {\S\S} \ref{sectionDeSing}--\ref{TangentCones} we construct a partial desingularisation of $\Snk$ near a well-behaved singular point $W$ and describe the tangent cone $C_W \Snk$.

For $2n = 2$, the $K$-valued symplectic bundles are precisely those of canonical determinant and, as outlined above, $\cS_{2, K}^k = B^k_{2, K}$ has been much studied. Our next objective is to answer some of the basic questions of nonemptiness, dimension and smoothness of $\Snk$ for $2n \ge 4$. In {\S} \ref{DimBounds}, we prove the following dimension bounds on various components of $\Snk$, generalising \cite[Theorem 3.4]{Baj19} of the first author.

\begingroup
\setcounter{tmp}{\value{theorem}}
\setcounter{theorem}{0}
\renewcommand\thetheorem{\Alph{theorem}}

\begin{theorem} \label{MainA} Let $C$ be any curve of genus $g \ge 2$.
\begin{enumerate}
\item[(a)] (Theorem \ref{FirstTypeBound}) Let $X$ be a closed irreducible sublocus of $\Snk$ of which a general element $W$ satisfies $H^0 ( C, W ) = H^0 ( C, L_W )$ where $L_W \subset W$ is a line subbundle of degree $d$. Then for each $W \in X$, we have
\[ \dim X \ \le \ \dim T_W X \ \le \ \dim \left( T_{L_W} B^k_{1, d} \right) + n(2n + 1)(g - 1) - 2nd - 1 . \]
\item[(b)] (Theorem \ref{SecondTypeBound}) Let $k$ be an integer satisfying $1 \le k \le n(g+1) - 1$. Suppose $Y$ is an irreducible component of $\Snk$ containing a bundle $W$ satisfying $h^0 ( C, W ) = k$ and such that the rank of the subbundle of $W$ generated by global sections is $r$. Then
\[ \dim Y \ \le \ \dim T_W Y \ \le \ \min \left\{ n(2n+1)(g-1) - ( 2k - 1 ) , n(2n+1)(g-1) - k - \frac{1}{2} r(r-1) \right\} . \]
\end{enumerate} \end{theorem}

In Corollary \ref{FirstTypeCpt}, we deduce some conditions on $g$, $n$ and $k$ for the existence of a \emph{component} $X$ of the form in Theorem \ref{MainA} (a).

In {\S} \ref{Nonemptiness} we construct stable symplectic bundles $W$ with prescribed values of $h^0 ( C, W )$, showing that $\Snk$ is nonempty in several cases. The approach is a combination of techniques from \cite{Mer} and \cite{CH3}: the $W$ we construct are ``almost split'' symplectic extensions $0 \to E \to W \to E^* \otimes K \to 0$ where $E$ and $K \otimes E^*$ are stable and have many sections. In {\S} \ref{Smoothness}, we show that if $C$ is Petri, in some cases $\Snk$ has a component which is smooth and of the expected dimension. To state the results, set
\begin{equation} k_0 \ := \ \max \{ k \ge 0 : \dim B^k_{1, g-1} \ge 1 \} . \label{Defnko} \end{equation}
By Brill--Noether theory, if $C$ is Petri then $k_0 = \left\lfloor \sqrt{g-1} \right\rfloor$, where $\lfloor t \rfloor = \max\{ m \in \bZ : m \le t \}$.

\begin{theorem} \label{MainB} Let $C$ be a curve of genus $g \ge 3$.
\begin{enumerate}
\item[(a)] (Theorem \ref{SnkNonempty}) For $1 \le k \le 2nk_0 - 3$, the locus $\Snk$ is nonempty.
\item[(b)] (Theorem \ref{GenSmoothExpDim}) If $C$ is a general Petri curve, then for $1 \le k \le 2nk_0 - 3$ there is a component of $\Snk$ which is generically smooth and of the expected codimension $\frac{1}{2}k(k+1)$.
\end{enumerate} \end{theorem}

\noindent We also briefly mention strictly semistable symplectic bundles in Remark \ref{Semistable}.

It should be noted that there are significantly stronger results in the rank two case. For $2n = 2$, the bound in Theorem \ref{MainB} (a) translates into $4 (g-1) \ge (k+3)^2$. For $g \ge 5$, Teixidor \cite{TiB04} showed for $4 (g-1) \ge k^2 - 1$ that $B^k_{2, K}$ is nonempty and has a component of the expected dimension, with a slightly better result for $k$ even. Furthermore, for $k \ge 8$ and $g$ prime, Lange, Newstead and Park \cite{LNP} showed that $B^k_{2, K}$ is nonempty for $4g - 4 \ge k^2 - k$. We certainly expect that the bound in Theorem \ref{MainB} can be improved for $2n \ge 4$.

In {\S} \ref{Superabundance}, we give an application of Theorem \ref{MainB} to standard Brill--Noether loci $B^k_{2n, 2n(g-1)}$. For $r \ge 2$, it was proven in \cite{TiB91} that in many cases $B^k_{r, d}$ has a component which is generically smooth and of the expected dimension. However, even for a generic curve, components of larger dimension can appear. Following \cite{CFK18}, we call such components \textsl{superabundant}. It was noted in \cite[{\S} 9]{New11} and \cite[\S 1]{BF} that $B^k_{2, K} \ = \ \cS_{2, K}^k$ in many cases (precisely; for $g < \frac{k(k - 1)}{2}$) has expected dimension strictly greater than $\beta^k_{2, 2g-2}$, despite the fact that $B^k_{2, K}$ is contained in $B^k_{2, 2g-2}$. For $n \ge 2$ it emerges that the expected dimension of $\Snk$ can also exceed $\beta^k_{2n, 2n(g-1)}$ for certain values of $g$, $n$ and $k$. We show the following.

\begin{theorem} \label{MainC} \quad
\begin{enumerate}
\item[(a)] (Theorem \ref{OversizedBNlocus1}) Suppose $m \ge 7$ and let $C$ be any curve of genus $g = m^2 + 1$. Then for any $n \ge 1$, the locus $\Sn^{2nm - 3}$ is nonempty and has dimension greater than $\beta^{2nm-3}_{2n, 2n(g-1)}$. In particular $B^{2nm-3}_{2n, 2n(g-1)}$ has a superabundant component.
\item[(b)] (Theorem \ref{OversizedBNlocus2}) Fix $n \ge 1$ and let $C$ be any curve of genus $g \ge (4n+7)^2 + 1$. For $k_0$ as defined in \ref{Defnko}, the locus $\Sn^{2nk_0 - 3}$ is nonempty and has dimension greater than $\beta^{2nk_0-3}_{2n, 2n(g-1)}$. In particular, $B^{2nk_0 - 3}_{2n, 2n(g-1)}$ has a superabundant component.
\end{enumerate} \end{theorem}
\endgroup

\noindent In {\S} \ref{SuperCohSys}, we also obtain similar results for certain moduli spaces of coherent systems, both with and without fixed determinant.

We note that Teixidor \cite{TiB04} also obtains superabundant components of $B^k_{2, K} = \cS^k_{2, K}$ for certain values of $k$.

Since Theorem \ref{MainC} (b) applies to all curves of genus $g \ge 122$, it gives a systematic way of finding ordinary determinantal varieties of dimension strictly greater than expected, in some ways akin to \cite[Proposition 9.1]{HHN}. We hope that this aspect of the present work may also be of interest outside the context of Brill--Noether theory.

The construction of the locus $\Snk$ is easily adapted for \emph{$K$-valued orthogonal} bundles; that is, bundles admitting a \emph{symmetric} $K$-valued bilinear form (see \cite{Mum}). However, our methods when applied to orthogonal bundles did not yield superabundant components of any $B^k_{r,d}$; and the argument of Theorem \ref{MainB} (b) also fails for orthogonal bundles. Therefore we have restricted our attention for the present to the symplectic case, with the intention of further studying orthogonal Brill--Noether loci in the future.

\subsection*{Acknowledgements}  We would like to thank Peter Newstead for helpful comments and for making us aware of several references.

\subsection*{Notation} Throughout, $C$ denotes a smooth projective curve of genus $g \ge 2$ over an algebraically closed field $\K$ of characteristic zero. For a sheaf $F$ over $C$, we abbreviate $H^i ( C, F )$, $h^i ( C, F )$ and $\chi ( C, F )$ to $H^i ( F )$, $h^i ( F )$ and $\chi ( F )$ respectively. If $A \times B$ is a product, we denote the projections by $\pi_A$ and $\pi_B$.

\section{Symplectic Brill--Noether loci}

\subsection{Moduli of \textit{K}-valued symplectic bundles} \label{background}
Let $W$ be a $K$-valued symplectic bundle of rank $2n$ over $C$. 
 By \cite[{\S} 2]{BG}, we have $\det (W) = K^n$. If $\kappa$ is a theta characteristic, then $V := W \otimes \kappa^{-1}$ is $\Oc$-valued symplectic. Thus $V$ is the associated vector bundle of a principal $\Sp_{2n}$-bundle $P$ over $C$. By a similar argument to that in \cite[{\S} 4]{Ram} (carried out in \cite{Hit0}), the vector bundle $V$ is stable if and only if $P$ is a \textsl{regularly stable} principal $\Sp_{2n}$-bundle; that is, stable and satisfying $\Aut (P) = Z(\Sp_{2n}) = \bZ_2$.

By \cite{Rthn}, there is a moduli space $\cM ( \Sp_{2n} )$ for stable principal $\Sp_{2n}$-bundles, which is an irreducible quasiprojective variety of dimension $n(2n+1)(g-1)$, and smooth at all regularly stable points. Moreover, it follows from \cite[Proposition 2.6 and Theorem 3.2]{Ser} that the natural map $\cM ( \Sp_{2n} ) \dashrightarrow \cU ( 2n, 0 )$ is an embedding. Translating by $\kappa$, we conclude:

\begin{lemma} The moduli space $\MS$ of stable vector bundles of rank $2n$ with $K$-valued symplectic structure is a smooth irreducible sublocus of $\Utn$, of dimension $n(2n+1)(g-1)$. \label{moduli} \end{lemma}

Furthermore, we recall a description of the tangent spaces of $\MS$. This can be obtained using Lie theory, but we offer a direct proof. It is well known that first order infinitesimal deformations of a vector bundle $W \to C$ are parametrised by $H^1 ( \End W )$. If $\omega \colon W \to W^* \otimes K$ is a skewsymmetric isomorphism, we have an identification
\begin{equation} \omega_* \colon H^1 ( \End W ) \ \isom \ H^1 ( K \otimes W^* \otimes W^* ) . \label{HOneIdentification} \end{equation}

\begin{lemma} \label{TWMS} Let $(W, \omega)$ be a $K$-valued symplectic bundle. Then the deformations of $W$ preserving the symplectic structure are parametrised by the subspace $H^1 ( K \otimes \Sym^2 W^* ) \subseteq H^1 ( K \otimes W^* \otimes W^* )$. In particular, if $W$ is stable, then $T_W \MS \cong H^1 ( K \otimes \Sym^2 W^* )$. \end{lemma}

\begin{proof} Let $0 \to W \to \bW \to W \to 0$ be a deformation of $W$, defining a class $\delta \in H^1 ( \End W )$. Consider the diagram
\[ \xymatrix{ 0 \ar[r] & W \ar[r] \ar[d]_\wr^\omega & \bW \ar[r] \ar@{..>}[d] & W \ar[r] \ar[d]_\wr^\omega & 0 \\
 0 \ar[r] & K \otimes W^* \ar[r] & K \otimes \bW^* \ar[r] & K \otimes W^* \ar[r] & 0 } \]
The class of the twisted dual deformation $K \otimes \bW^*$ is $- \Iden_K \otimes {^t \delta}$. The maps patch together to give a deformation of $\omega$ if and only if
\[ \omega^* (- \Iden_K \otimes {^t \delta}) \ = \ \omega_* \delta \ \in \ H^1 (\Hom (W, K \otimes W^*) ) . \]
Now using the fact that $\Iden_K \otimes \omega = - {^t\omega}$, we compute for any $\delta \in H^1 ( \End W )$ that
\[ \omega^* (- \Iden_K \otimes {^t \delta}) \ = \
\Iden_K \otimes {^t ( \omega_* \delta )} . \]
Therefore, $\omega$ can be lifted to the extension $\delta$ if and only if $\omega_* \delta = \Iden_K \otimes {^t ( \omega_* \delta )}$; that is $\omega_* \delta \in H^1 ( K \otimes \Sym^2 W^* )$. \end{proof}

\subsection{The scheme structure of symplectic Brill--Noether loci} \label{SchemeStructure}

As already noted, bundles of rank two and canonical determinant are precisely the $K$-valued symplectic bundles of rank two. We shall see that the construction of $B^k_{2, K} = \cS_{2, K}^k$ in \cite{Muk92} and \cite{Muk97} generalises virtually word for word to higher rank $K$-valued symplectic bundles.

To construct $\Snk$ as a scheme, we require a suitable Poincar\'e bundle equipped with a family of symplectic forms. As $\MS \cong \cM ( \Sp_{2n} )$ and the group $\Sp_{2n}$ is not of adjoint type, by \cite{BBNN} there is no Poincar\'e bundle over $\MS \times C$. The following lemma shows that Poincar\'e bundles do exist over small enough \'etale open subsets of $\MS$.

\begin{lemma} \label{Poincare} There exists an \'etale open covering $\{ U_\alpha \}$ of $\MS$, together with Poincar\'e bundles $\cW_\alpha \to U_\alpha \times C$, each equipped with a family $\omega_\alpha \colon \cW_\alpha \otimes \cW_\alpha \to \pi_C^* K$ of symplectic forms. \end{lemma}

\begin{proof} As $\MS$ is contained in $\Utn$, there exists an \'etale cover $\tM \to \MS$ together with a Poincar\'e bundle $\cW \to \tM \times C$. By stability, for any $W \in \MS$ we have $h^0 ( K \otimes \wedge^2 W^* ) = 1$. Hence by \cite[Corollary III.12.9]{Har}, the sheaf
\[ \cB \ := \ ( \pi_{\tM} )_* \left( \pi_C^* K \otimes \wedge^2 \cW^* \right) \]
is locally free of rank one over $\tM$. Let $\{ U_\alpha \}$ be an open covering of $\tM$ such that $\cB|_{U_\alpha}$ is trivial for each $\alpha$. Now if $W$ is a stable vector bundle of slope $g-1$, then any nonzero map $W \to W^* \otimes K$ is an isomorphism. Therefore, any generating section $\omega_\alpha$ for $\cB|_{U_\alpha}$ defines a family of symplectic structures on $\cW_\alpha := \cW|_{U_\alpha \times C}$. The lemma follows. \end{proof}

We proceed to study the \emph{symmetric determinantal} structure of $\Snk$.  The following proposition is an obvious generalisation of \cite[Theorem 4.2]{Muk92}, and is essentially contained in \cite[Remark 4.6]{Muk92}. We give the proof, because the construction will be used further in {\S\S} \ref{sectionDeSing}--\ref{TangentCones}. 


\begin{proposition} \label{SchemeStructureSkW} \quad
\begin{enumerate}
\item[(a)] Scheme-theoretically, $\Snk$ is \'etale locally defined by the vanishing of the $(\nu - k + 1) \times (\nu - k + 1)$-minors of a $\nu \times \nu$ symmetric matrix, for some $\nu \ge k$.
\item[(b)] Each component of $\Snk$ is of codimension at most $\frac{1}{2}k(k+1)$.
\item[(c)] The sublocus $\Sn^{k+1}$ is contained in $\Sing ( \Snk )$.
\end{enumerate}
\end{proposition}

\begin{proof} (a) We begin with a slightly more general situation. Let $\cW \to S \times C$ be a family of bundles of rank $2n$ over $C$, and let $\omega \colon \cW \otimes \cW \to \pi_C^*K$ be a family of $K$-valued symplectic structures on $\cW$. For $k \ge 0$, we define the Brill--Noether locus associated to the family $\cW$ set-theoretically as
\[ \SkW \ := \ \{ s \in S : h^0 ( C, \cW_s ) \ge k \} . \]
Now for any effective divisor $D$ on $C$, the coherent sheaf
\begin{align} \label{diag6}
\cF \ := \ \left( \pi_S \right)_* \left( \frac{\cW \otimes \pi_C^*\Oc(D)}{\cW \otimes \pi_C^*\Oc (-D)} \right) ,
\end{align}
is locally free of rank $4n \cdot \deg(D)$ over $S$. We shall define a symplectic structure on $\cF$. We extend $\omega$ linearly over $\pi_C^*\Oc$ to a symplectic form
\[ \wedge^2 \left( \cW \otimes \pi_C^* \Oc (D) \right) \ \to \ \pi_C^* K(2D) . \]
Now $\omega_s ( \cW_s (-D), \cW_s(D) ) \subseteq K$ for all $s$. Thus, if $t, u$ are elements of $\cF_s = H^0 \left( C, \frac{\cW_s (D)}{\cW_s(-D)} \right)$ and $\Res$ is the residue map, then
\[ \sum_{x \in \Supp(D)} \Res \left( \omega_s ( t_x , u_x) \right) \ =: \ \bomega_s (t, u) \]
is a well-defined element of $H^1 ( K )$. Thus $\omega$ descends to a bilinear map
\[ \bomega \colon \wedge^2 \cF \ \to \ \cO_S \otimes H^1 ( K ) \ = \ \cO_S . \]
Moreover, $\bomega$ is nondegenerate since $\omega$ is.

Let us now assume that $\deg(D)$ is large enough that $h^1 ( C, \cW_s (D) ) = 0$ for all $s \in S$. Then, as $\cW_s \cong \cW^*_s \otimes K$, by Serre duality $h^0 ( C, \cW_s (-D) ) = 0$ for all $s \in S$ also. Thus the subsheaf
\[ \cL_1 \ := \left( \pi_S \right)_* \left( \frac{\cW}{\cW \otimes \pi_C^*\Oc(-D)} \right) \ \subset \ \cF \]
is locally free of rank $2n \cdot \deg(D)$. As the residue of a regular differential is zero, $\cL_1$ is Lagrangian with respect to $\bomega$.

Furthermore, as $h^1 ( \cW_s (D) ) = 0$ for all $s$, the subsheaf
\[ \cL_2 \ := \ \im \left( ( \pi_S )_* \left( \cW \otimes \pi_C^* \Oc ( D ) \right) \ \to \ \cF \right) \ \subset \ \cF \]
is also locally free of rank $2n \cdot \deg (D)$. By the residue theorem \cite[III.7.14.2]{Har}, in fact $\cL_2$ also defines a Lagrangian subbundle of $\cF$. Moreover, it is easy to see that $\cL_1|_s \cap \cL_2|_s \cong H^0 ( C, \cW_s )$ for each $s \in S$, so
\begin{equation} \SkW \ = \ \{ s \in S : \dim \left( \cL_1|_s \cap \cL_2|_s \right) \ge k \} . \label{SkWasDegenLocus} \end{equation}

Now let $U \subseteq S$ be an open set over which $\cF$ is trivial. Then any choice of Lagrangian subbundle of $\cF|_U$ complementary to $\cL_1|_U$ defines a local splitting $\cF|_U \isom \cL_1|_U \oplus \cL_1^*|_U$. Perturbing this choice and shrinking $U$ if necessary, we can assume in addition that $\cL_1^*|_s \cap \cL_2|_s = 0$ for all $s \in U$. Then, as in \cite[Examples 1.5 and 1.7]{Muk97}, there exists a \emph{symmetric} map $\Sigma_U \colon \cL_1|_U \to \cL_1^*|_U$ with the property that $\cL_2|_U$ is the graph of $\Sigma_U$, and for each $s \in U$ moreover
\[ \Ker ( \Sigma_U|_s ) \ = \ \cL_1|_s \cap \cL_2|_s . \]
It follows by (\ref{SkWasDegenLocus}) that $\SkW \cap U$ is defined by the condition $\rank ( \Sigma_U|_s ) \le 2n \cdot \deg (D) - k$, so is cut out by the vanishing of the $(\nu - k + 1) \times (\nu - k + 1)$-minors of a local matrix expression for $\Sigma_U$, where $\nu = \rank (\cL_1) = 2n \cdot \deg (D)$. Clearly, $S$ can be covered by such open sets $U$.

Now we specialise to $S = U_\alpha$ and $(\cW, \omega) = ( \cW_\alpha, \omega_\alpha )$ as defined in Lemma \ref{Poincare}. Statement (a) follows as $\Snk$ is the union of the images of the loci $\cS^k ( \cW_\alpha )$ by an \'etale map.

Parts (b) and (c) follow from part (a), by general properties of symmetric determinantal loci. (In fact these statements are true for any family $\cW \to S \times C$ of $K$-valued symplectic bundles.) \end{proof}

\begin{remark} In \cite{Oss}, \cite{Oss-2} and \cite{Zha} the above approach is generalised to the setting of \emph{multiply symplectic Grassmannians} and used to give lower bounds on fixed determinant Brill--Noether loci $B^k_{r, L}$ for special line bundles $L$. \end{remark}

\subsection{Tangent spaces of symplectic Brill--Noether loci} \label{TangentSpaces}

Let us now describe the Zariski tangent spaces of $\Snk$, following the discussion for bundles of rank two in \cite[{\S} 1]{TiB07}. Firstly, we require a definition. Recall that for any bundle $W \to C$ we have the \textsl{Petri map}
\[ \mu \colon H^0 ( W ) \otimes H^0 ( K \otimes W^* ) \ \to \ H^0 ( K \otimes \End W ) . \]
If $\omega \colon W \isom K \otimes W^*$ is an isomorphism, then we obtain an identification of the Petri map with the multiplication map
\begin{equation} H^0 ( W ) \otimes H^0 ( W ) \ \to \ H^0 ( W \otimes W ) , \label{SelfDualPetri} \end{equation}
If $W$ is simple (for example, stable) then this identification is canonical up to scalar. In this case, we abuse notation slightly and denote the map (\ref{SelfDualPetri}) also by $\mu$. Clearly, $\mu \left( \Sym^2 H^0 ( W ) \right) \subseteq H^0 ( \Sym^2 W )$. Let $\sym \colon H^0 ( W ) \otimes H^0 ( W ) \to \Sym^2 H^0 ( W )$ be the canonical surjection.

\begin{definition} \label{DefnMusL} Let $W \to C$ be a $K$-valued symplectic bundle. For any subspace $\Lambda \subseteq H^0 ( W )$, we write
\[ \musL \colon \sym ( \Lambda \otimes H^0 ( W )) \ \to \ H^0 ( \Sym^2 W ) \]
for the restriction of (\ref{SelfDualPetri}). We abbreviate $\mus_{H^0 ( W )}$ to $\mus$. Furthermore, for any subspace $\Pi$ of $H^0 ( W \otimes W )$ we write
\[ \Pi^\perp \ := \ \{ v \in H^1 ( K \otimes \Sym^2 W^* ) : v \cup \Pi = 0 \} , \]
the orthogonal complement of $\Pi$ in $H^1 ( K \otimes \Sym^2 W )$. \end{definition}

\begin{proposition} \label{presLambda} Let $W$ be a simple $K$-valued symplectic bundle. For any subspace $\Lambda \subseteq H^0 ( W )$, the space of first-order infinitesimal deformations preserving $\Lambda$ is exactly $\im ( \musL )^\perp$.
\end{proposition}

\begin{proof} As in the proof of \cite[Proposition IV.4.1]{ACGH}, using also the identification (\ref{HOneIdentification}), one shows that the space of first-order infinitesimal deformations of the vector bundle $W$ which preserve the subspace $\Lambda$ is given by
\[ \{ v \in H^1 ( K \otimes W^* \otimes W^* ) : v \cup \mu ( \Lambda \otimes H^0 ( W ) ) = 0 \} , \]
the orthogonal complement of $\mu \left( \Lambda \otimes H^0 ( W ) \right)$ in the full deformation space $H^1 ( K \otimes W^* \otimes W^* )$. Thus we must describe the intersection of this space with $H^1 ( K \otimes \Sym^2 W^* )$.

Suppose $v \in H^1 ( K \otimes \Sym^2 W^* )$. Then clearly $v \cup \mu \left( \wedge^2 H^0 ( W ) \right) = 0$, whence
\[ v \cup \mu ( \sigma ) \ = \ v \cup \mu ( \sym ( \sigma ) ) \]
for all $\sigma \in H^0 ( W ) \otimes H^0 ( W )$. It follows, as desired, that
\[ \mu ( \Lambda \otimes H^0 ( W ) )^\perp \ = \ \mu \circ \sym ( \Lambda \otimes H^0 ( W ) )^\perp \ = \ \im ( \musL )^\perp \ \subseteq \ H^1 ( K \otimes \Sym^2 W^* ) . \qedhere \]
\end{proof}

\begin{corollary} \label{SmoothCrit} Suppose $W$ is a stable $K$-valued symplectic bundle with $h^0 ( W ) = k$. Then $\Snk$ is smooth and of codimension $\frac{1}{2} k (k+1)$ at $W$ if and only if $\mus \colon \Sym^2 H^0 ( W ) \ \to \ H^0 ( \Sym^2 W )$ is injective. \end{corollary}

\begin{proof} By Proposition \ref{presLambda}, we have $T_W \Snk = \im ( \mus )^\perp$. Now clearly
\[ \dim \im ( \mus )^\perp \ = \ \dim \MS - \dim \Sym^2 H^0 ( W ) + \dim \Ker ( \mus ) . \]
Since $\Sym^2 H^0 ( W )$ has dimension $\frac{1}{2}k(k+1)$, we see that $T_W \Snk$ has the expected codimension if and only if $\mus$ is injective. \end{proof}

\subsection{Desingularisations of symplectic Brill--Noether loci} \label{sectionDeSing}

In this subsection, we adapt arguments for determinantal varieties from \cite{ACGH} to construct a partial desingularisation of (an \'etale cover of) the symplectic Brill--Noether stratum $\Snk$, and use it to obtain information on smooth points of lower strata. In the next section, we shall also use the desingularisation to study the tangent cones of $\Snk$. This approach was used in a similar way in \cite{ACGH}, \cite{CTiB} and \cite{HHN} for the study of, respectively, Brill--Noether loci in $\Pic(C)$, higher rank Brill--Noether loci $B^k_{r, d}$ and twisted Brill--Noether loci $B^k_{n, e} (V)$.

Let $W$ be a stable $K$-valued symplectic bundle with $h^0 ( W ) \ge k \ge 1$. By Lemma \ref{Poincare} and Proposition \ref{SchemeStructureSkW} (a), we can find an \'etale neighbourhood $S$ of $W$ in $\MS$ and a Poincar\'e bundle $\cW \to S \times C$, together with a symmetric map of vector bundles $\Sigma \colon \cL_1 \to \cL_1^*$ over $S$ such that for each $s \in S$ we have $\Ker \left( \Sigma_s \right) \cong H^0 ( \cW_s )$, so
\[ \Snk \times_{\MS} S \ = \ \SkW \ = \ \{ s \in S : \dim \Ker ( \Sigma|_s ) \ge k \} , \]
an \'etale cover of $\Snk$ near $W$.

We consider the Grassmann bundle $\Gr ( k , \cL_1 )$ parametrising $k$-dimensional linear subspaces of fibres of $\cL_1$. In analogy with \cite[IV.3]{ACGH}, we define
\begin{equation} \SGkW \ := \ \{ \Lambda \in \Gr ( k, \cL_1 ) : \Sigma ( \Lambda ) = 0 \} . \label{DefnSkW} \end{equation}
A point of $\SGkW$ is a pair $( \cW_s , \Lambda )$ where $\cW_s$ is a symplectic bundle represented in $S$ and $\Lambda$ a $k$-dimensional subspace of $H^0 ( \cW_s )$. Such a pair will be called a \emph{symplectic coherent system}. We write $c \colon \SGkW \to S$ for the projection.

\begin{theorem} \label{DeSing} Let $W$, $S$, $\cW$ and $\Sigma \colon \cL_1 \to \cL_1^*$ be as above, and suppose that $\Lambda \subseteq H^0 ( W )$ is a subspace of dimension $k$.
\begin{enumerate}
\item[(a)] The tangent space to $\SGkW$ at $(W, \Lambda)$ fits into an exact sequence
\begin{equation} \label{TangSpGk}
0 \ \to \ \Hom( \Lambda, H^0 ( W ) / \Lambda ) \ \to \ T_{(W, \Lambda)} \SGkW \ \xrightarrow{c_*} \ T_W \MS .
\end{equation}
The image of the differential $c_*$ coincides with $\im ( \musL )^\perp$ (cf.\ Definition \ref{DefnMusL}).
\item[(b)] The locus $\SGkW$ is smooth and of dimension $\dim \MS - \frac{1}{2}k(k+1)$ at $(W, \Lambda)$ if and only if $\musL$ is injective.
\item[(c)] Suppose $\musL$ is injective for all $\Lambda \in \Gr(k, H^0 ( W ) )$. Then $\SGkW$ is smooth in a neighbourhood of $c^{-1}(W)$, and $c^{-1}(W)$ is a smooth scheme. In particular, in this case $\SGkW$ contains a desingularisation of a neighbourhood of $W$ in $\SkW$. Furthermore, the normal space $N := N_{c^{-1}(W)/\SGkW}$ is precisely
\[ \{ (\Lambda, v) : v \cup \im ( \mus_\Lambda ) \ = \ 0 \} \ \subset \ \Gr(k, H^0 (W) ) \times H^1 ( K \otimes \Sym^2 W^* ) , \]
and the differential $c_* \colon N \to T_W \MS$ is the projection to the second factor.
\end{enumerate}
\end{theorem}

\begin{proof}
(a) By the construction of $\SGkW$, we have
\[ c^{-1} ( W ) \ = \ \Gr ( k, H^0 ( W ) ) . \]
Therefore, $\Ker ( c_* ) \cong T_\Lambda \Gr ( k, H^0 ( W ) ) \cong \Hom ( \Lambda, H^0 ( W ) / \Lambda ))$. For the rest: Exactly as in the line bundle case \cite[Proposition IV.4.1 (ii)]{ACGH}, the image of $c_*$ is the space of tangent vectors in $T_s S = T_W \MS$ preserving the subspace $\Lambda$. By Proposition \ref{presLambda}, this is exactly $\im ( \musL )^\perp$.

(b) Note that $\left( \Lambda \otimes H^0 ( W ) \right) \cap \Ker ( \sym ) = \wedge^2 \Lambda$. Therefore,
\[ \dim \left( \sym \left( \Lambda \otimes H^0 ( W ) \right) \right) \ = \ \dim \left( \Lambda \otimes H^0 ( W ) \right) - \dim \left( \wedge^2 \Lambda \right) \ = \
 k \cdot h^0 ( W ) - \frac{k(k-1)}{2} . \]
By part (a), the dimension of $T_\Lambda \SGkW$ is given by
\begin{multline*} 
 k (h^0 ( W ) - k) + \dim \MS - \dim \sym \left( \Lambda \otimes H^0 ( W ) \right) + \dim \ker ( \mu_\Lambda^s ) \ = \\
 \dim \MS - \frac{k(k+1)}{2} + \dim \ker ( \mu_\Lambda^s ) . \end{multline*}
Part (b) follows. All statements in part (c) are immediate consequences of part (a). \end{proof}

\noindent The first application of Theorem \ref{DeSing} is very similar to \cite[Proposition 3.12]{HHN}:

\begin{lemma} \label{OneImpliesMore} Suppose $\Snk$ has a component $X$ which is generically smooth of the expected codimension $\frac{1}{2} k (k+1)$. Then for $1 \le \ell \le k$, the component $X$ lies in a component of $\Sn^\ell$ which is generically smooth and of the expected codimension $\frac{1}{2} \ell (\ell + 1)$. \end{lemma}

\begin{proof} By induction, it suffices to prove this for $\ell = k-1$, where $k \ge 2$. Let $W$ be a smooth point of $X$, so $h^0 ( W ) = k$ and $\mus \colon \Sym^2 H^0 ( W ) \to H^0 ( \Sym^2 W )$ is injective. Define $SG^{k-1} ( \cW )$ as in (\ref{DefnSkW}) in an \'etale neighbourhood of the present $W$. By hypothesis and Theorem \ref{DeSing} (b), for any $\Lambda \subset H^0 ( W )$ of dimension $k-1$, the space $SG^{k-1} ( \cW )$ constructed above is smooth of codimension $\frac{1}{2}k(k-1)$ at $( W, \Lambda )$. Thus $(W, \Lambda)$ lies in a component $\tY_{k-1}$ of $SG^{k-1} ( \cW )$ which is generically smooth and of the expected codimension $\frac{1}{2}k(k-1)$ at $( W, \Lambda )$. Now the inverse image of $\Snk$ in $\tY_{k-1}$ has dimension at most
\[ \dim X + \dim \Gr ( k-1 , k ) \ = \ \left( \dim \MS - \frac{k(k-1)}{2} \right) - 1 , \]
which is less than $\dim \SG^{k-1} (\cW)$. Therefore, a general $(W', \Lambda') \in \tY_{k-1}$ is smooth and satisfies $h^0 ( W' ) = k-1$. It follows that the image of $SG^{k-1} ( \cW )$ in $\Sn^{k-1}$ lies in a component which is generically smooth and of the expected codimension. The statement follows. \end{proof}

\subsection{Tangent cones of symplectic Brill--Noether loci} \label{TangentCones}

We shall now describe the tangent cone $C_W \Snk$ at a ``well-behaved'' singular point $W$. We begin by adapting \cite[Lemma, p.\ 242]{ACGH} for symmetric determinantal varieties. Let $A$ and $\bE$ be vector spaces of dimensions $a$ and $\be$ respectively, and let $\bphi \colon \Sym^2 A \to \bE$ be a linear map. As before, write $\sym \colon A \otimes A \to \Sym^2 A$ for the canonical surjection. 
 Let $\{ \alpha_1, \ldots , \alpha_a \}$ be a basis of $A$, and write $\xij := \bphi \circ \sym (\alpha_i \otimes \alpha_j )$. (Note that the symbol $x_i$ has not been defined.)

\begin{lemma} \label{ACGHsymm} Assume that $\bphi_{\Lambda} := \bphi|_{\sym(\Lambda \otimes A)}$ is injective for each $\Lambda \in \Gr(k, A)$. Set
\[ \bI \ := \ \left\{ (\Lambda, v) \in \Gr(k, A) \times \bE^* : v \in \bphi \left( \sym ( \Lambda \otimes A) \right)^\perp \right\} . \]
Let $\bp \colon \Gr(k, A) \times \bE^* \to \bE^*$ denote the projection. Then the following holds.
\begin{enumerate}
\item[(a)] The scheme $\bp (\bI)$ is Cohen--Macaulay, reduced and normal.
\item[(b)] The ideal of $\bp (\bI)$ is generated by the $( a - k + 1 ) \times ( a - k + 1 )$ minors of the symmetric matrix $(\xij )_{i, j = 1, \ldots, a}$.
\item[(c)] The degree of $\bp(\bI)$ is
\[ \prod_{i = 0}^{a - k + 1} \frac{\binom{a + i}{a - k - i}}{\binom{2i + 1}{i}} . \]
\item[(d)] The morphism $\bp$ maps $\bI$ birationally onto $\bp (\bI)$.
\end{enumerate} \end{lemma}

\begin{proof} As this follows very closely the proof of \cite[Lemma, p.\ 242]{ACGH}, we give only a sketch. The injectivity hypothesis implies that $\bI$ is a vector bundle over $\Gr ( k, A )$ which is smooth of dimension $\be - \frac{k(k+1)}{2}$. Let $\bJ$ be the subvariety of $\bE^*$ whose ideal is generated by the $(a - k + 1) \times (a - k + 1)$ minors of the symmetric matrix $(x_i x_j)_{i, j = 1, \ldots , a}$. As in the proof of loc.\ cit., we see that $\bJ$ is supported exactly on $\bp (\bI)$. Hence they coincide scheme-theoretically and $\bJ$ is a symmetric determinantal variety of the expected dimension. Thus $\bJ$ is Cohen--Macaulay by \cite[Theorem 1.2.14]{Miro}. The proofs of (a), (b) and (d) now follow verbatim those of loc.\ cit.\ (i), (ii) and (iv) respectively. As for (c): Note that $\bJ = \bp (\bI)$ is the pullback of
\[ \{ M \in \Sym^2 \K^a : \dim \Ker (M) \ge k \} \]
by the map $\bE^* \to \Sym^2 \K^a$ given by $v \mapsto \begin{pmatrix} x_i x_j (v) \end{pmatrix}$. As this map is linear and $\bJ$ is of the expected codimension, the statement follows directly from \cite[p.\ 78]{HT}. \end{proof}

\begin{theorem} \label{TangentConeSnk} Suppose $W \in \Snk$ is such that for all $\Lambda \in \Gr(k, H^0 ( W ) )$, the map $\musL$ is injective. Let $\alpha_1 , \ldots , \alpha_{h^0 (W)}$ be a basis for $H^0 (W)$, and define $x_i x_j$ as above.
\begin{enumerate}
\item[(a)] As sets, we have
\[ C_W \Snk \ = \ \bigcup_{\Lambda \in \Gr( k, H^0 ( W ) )} \im ( \mu^s_\Lambda )^\perp . \]
\item[(b)] The tangent cone $C_W \Snk$ to $\Snk$ at $W$ is Cohen--Macaulay, reduced and normal.
\item[(c)] The ideal of $C_W \Snk$ as a subvariety of $H^1 ( K \otimes \Sym^2 W^* )$ is generated by the $( h^0 ( W ) - k + 1) \times ( h^0 ( W ) - k + 1 )$-minors of the symmetric matrix $(\xij)_{i,j=1, \ldots , h^0 ( W )}$.
\item[(d)] The multiplicity of $\Snk$ at $W$ is
\[ \prod_{i = 0}^{h^0 ( W ) - k + 1} \frac{{h^0 ( W ) + i \choose h^0 ( W ) - k - i}}{{2i + 1 \choose i}} . \]
\end{enumerate}
\end{theorem}

\begin{proof} By Theorem \ref{DeSing} (c) and Lemma \ref{ACGHsymm} (a) \& (d), the hypotheses of \cite[Lemma II.2.1.3, p.\ 66]{ACGH} are satisfied by the map $\bp \colon \bI \to \bE^*$.
 Therefore, $\bp ( \bI )$ coincides scheme-theoretically with $C_W \Snk$. Part (a) follows immediately from the definition of $\bp$. Parts (b), (c) and (d) follow from Lemma \ref{ACGHsymm} (a), (b) and (c) respectively. \end{proof}

\section{Dimension bounds on symplectic Brill--Noether loci} \label{DimBounds}

We begin this section with an important result on the structure of bundles with nonvanishing sections.

\begin{lemma} \label{FeinbergLemma} Let $V$ be a vector bundle over $C$ with $h^0 ( V ) \ge 1$. Let $B \subset C$ be the subscheme of $C$ along which all sections of $V$ vanish. Its support is the finite set
\[ \{ p \in C : s (p) = 0 \hbox{ for all } s \in H^0 ( V ) \} . \]
If the subbundle $E \subseteq V$ generated by global sections is of rank at least two, then there exists a section of $V$ which is nonzero at all points of $C \backslash \Supp (B)$. \end{lemma}

\begin{proof} This is \cite[Proposition 1]{Baj19}, whose proof is due to Feinberg \cite{Fein} (see \cite{TiB92}). \end{proof}

\begin{corollary} \label{FeinbergCorollary} Any vector bundle $V$ with $h^0 ( V ) \ge 1$ can be written as an extension $0 \to \Oc (D) \to V \to F \to 0$ where $D$ is effective and $H^0 ( \Oc (D) ) = H^0 ( V )$ or $h^0 ( \Oc (D) ) = 1$. \end{corollary}

\noindent Motivated by Corollary \ref{FeinbergCorollary}, we recall \cite[Definition 1]{Baj19}:

\begin{definition} \label{TwoTypes} A vector bundle $V$ over $C$ with $h^0 ( V ) \ge 1$ will be said to be \textsl{of first type} if $V$ contains a line subbundle $L$ such that $H^0 ( V ) = H^0 ( L )$. If $V$ contains a line subbundle $L$ with $h^0 ( L ) = 1$, then $V$ is said to be \textsl{of second type}. Note that if $h^0 ( V ) = 1$ then $V$ is both of first type and of second type. \end{definition}

The relevance of this for higher rank Brill--Noether loci is illustrated by \cite[Theorem 1.1]{CFK18}, which states that for $3 \le \nu \le \frac{g+8}{4}$, if $C$ is a general $\nu$-gonal curve then $B^2_{2, d}$ has two components, corresponding to the two types in Definition \ref{TwoTypes}. In a similar way, we shall see that different dimension bounds apply for components of $\Snk$ whose generic elements are of different types.

We shall require the following technical lemma in several places.

\begin{lemma} \label{jstar} Let $V$ be any vector bundle, and let $0 \to M \xrightarrow{\iota} K \otimes V^* \to G \to 0$ be an extension where $M$ has rank one. Consider the induced map
\[ \iota^* \colon K^{-1} \otimes V \otimes V \ \to \ M^{-1} \otimes V . \]
Then the restriction of $\iota^*$ to $K^{-1} \otimes \Sym^2 V$ is surjective.
\end{lemma}

\begin{proof} Let $\lambda$ be a local generator for $M^{-1}$, and let $\phi_1 , \ldots , \phi_n$ be a local basis of $V$ such that no $\phi_i$ belongs to the corank one subbundle $G^* \subset V$. Since $M$ has rank one, for any local generator $\delta$ of $K^{-1}$ the section $\iota^* ( \delta \otimes \phi_i \otimes \phi_i ) = {^t\iota (\delta \otimes \phi_i)} \otimes \phi_i$ is proportional to $\lambda \otimes \phi_i$. It follows that the image of $\iota^*|_{K^{-1} \otimes \Sym^2 V}$ contains a spanning set of local sections near any point. This proves the lemma. \end{proof}

By the Clifford theorem for stable vector bundles \cite{BGN}, for all stable $K$-valued symplectic bundles $W$ of rank $2n$ we have $h^0 (W) \le n(g+1) - 1$. In what follows, we shall assume $0 \le k \le n(g+1) - 1$.

\subsection{Symplectic bundles of first type} \label{subsection:FirstType}

\begin{theorem} \quad \label{FirstTypeBound} Let $X$ be a closed irreducible sublocus of $\Snk$ of which a general element $W$ satisfies $h^0 ( W ) = h^0 ( L_W ) = k$ for a line subbundle $L_W \subset W$ of degree $d$. For such $W$, we have
\[ \dim X \ \le \ \dim \left( T_W X \right) \ \le \ \dim \left( T_{L_W} B^k_{1, d} \right) + n(2n + 1)(g - 1) - 2nd - 1 . \]
\end{theorem}

\begin{proof} The inclusion $j \colon L \to W$ induces maps on cohomology
\[ j^* \colon H^1 ( \End (W) ) \to H^1 ( \Hom ( L, W ) ) \quad \hbox{and} \quad j_* \colon H^1 ( \End (L) ) \to H^1 ( \Hom ( L, W ) ) . \]
A deformation $\bW$ of $W$ induces a given deformation $\bL$ of the subbundle $L$ if and only if there is a commutative diagram
\[ \xymatrix{ 0 \ar[r] & W \ar[r] & \bW \ar[r] & W \ar[r] & 0 \\
 0 \ar[r] & L \ar[r] \ar[u]^j & \bL \ar[r] \ar[u] & L \ar[r] \ar[u]^j & 0 . } \]
This is equivalent to the condition
\begin{equation} j^* \delta( \bW ) \ = \ j_* \delta ( \bL ) \hbox{ in } H^1 ( \Hom ( L, W ) ) . \label{CohomCrit} \end{equation}
Now $L$ defines a point of $B^k_{1, d}$. The deformation $\bW$ corresponds to a tangent direction in $T_W X$ if and only if $\bW$ satisfies (\ref{CohomCrit}) for some $\bL$ belonging to $T_L B^k_{1,d} \subseteq H^1 ( \End (L) )$. It follows that
\begin{equation} T_W X \ = \ ( j^* )^{-1} j_* \left( T_L B^k_{1,d} \right) . \label{TypeOneTangSp} \end{equation}

Composing with $\omega \colon W \isom K \otimes W^*$, we view $j$ as a map $L \to K \otimes W^*$, and then
\[ j^* \colon H^1 ( K^{-1} \otimes W \otimes W ) \ \to \ H^1 ( L^{-1} \otimes W ) . \]
By Lemma \ref{jstar}, the restriction of $j^*$ to the subspace
\[ H^1 ( K^{-1} \otimes \Sym^2 W ) \ \isom \ H^1 ( K \otimes \Sym^2 W^* ) \ = \ T_W \MS \]
remains surjective (the first identification above is given by $\omega \otimes \omega$).
 By this fact and (\ref{TypeOneTangSp}), we have
\begin{equation} \dim ( T_W X ) \ \le \ \dim ( T_L B^k_{1,d} ) + h^1 ( K \otimes \Sym^2 W^* ) - h^1 ( L^{-1} \otimes W ) . \label{TWXineq} \end{equation}
Now as $W$ is of first type, there can be at most one independent vector bundle injection $L \to W$, so $h^0 ( L^{-1} \otimes W ) = 1$. Then by Riemann--Roch,
\[ h^1 ( L^{-1} \otimes W ) \ = \ 1 - \chi ( L^{-1} \otimes W ) 
 \ = \ 1 + 2nd . \]
As moreover $h^1 ( K \otimes \Sym^2 W^* ) = n(2n+1)(g-1)$, the theorem follows from (\ref{TWXineq}).
\end{proof}

For $k = 1$, Theorem \ref{FirstTypeBound} together with the codimension condition gives the familiar fact that the set of bundles with sections is a divisor. Moreover, if $W$ is a general bundle with one independent section then this section does not vanish, as if $X$ is a locus as in the theorem with $k = 1$ and $d \ge 1$ then $X$ has codimension at least $(2n - 1)d + 1 \ge 2$. More generally, Theorem \ref{FirstTypeBound} gives the following restrictions on the parameter $n$ for \emph{components} in $\Snk$ whose general element is of first type.

\begin{corollary} \quad \label{FirstTypeCpt}
\begin{enumerate}
\item[(a)] Suppose $n \geq 1$ and $k \geq 2$. 
 Then $\Snk$ has a component whose generic element $W$ satisfies $H^0 ( W ) = H^0 ( L_W )$ for a degree $d$ line subbundle only if $8n - 2 \le k$. In particular, for all $n \ge 1$, the generic element of any component of $\Sn^2$ is of second type.
\item[(b)] Suppose $d \ge 1$. Then $\Snk$ has a component whose generic element $W$ satisfies $H^0 ( W ) = H^0 ( L_W )$ for a degree $d$ line subbundle $L_W$ only if $n \le \frac{g+4}{16}$.
\end{enumerate}
\end{corollary}

\begin{proof} (a) Let $W$ be a general point of a component as in the statement. As any component of $\Snk$ has codimension at most $\frac{1}{2}k(k+1)$, by Theorem \ref{FirstTypeBound} we have
\begin{equation} 2nd \ \le \ \dim \left( T_{L_W} B^k_{1,d} \right) + \frac{k(k+1)}{2} - 1 , \label{FirstTypeComponent} \end{equation}
By Martens' theorem \cite[p.\ 191 ff.]{ACGH}, and noting that the usual Martens bound is in fact a bound for $\dim \left( T_{L_W} B^k_{1,d} \right)$, we have $\dim \left( T_{L_W} B^k_{1,d} \right) \le d - 2(k - 1)$. Thus the above inequality becomes
\[ (2n - 1) d \ \le \ \frac{k(k+1)}{2} - 2k + 1 \ = \ \frac{(k-1)(k-2)}{2} . \]
By Clifford's theorem \cite[p.\ 107 ff.]{ACGH} applied to the line bundle $L_W$, we have $k \le \frac{d}{2} + 1$. Using this and the fact that $d \ne 0$ since $k = h^0 ( L_W ) \ge 2$, the above inequality becomes
\[ 2n - 1 \ \le \ \frac{\frac{d}{2} \cdot ( k-2 )}{2d} \ = \ \frac{k-2}{4}, \]
which gives 
 $8n - 2 \le k$, as desired.

(b) Suppose $X$ is a component as in the statement. As in part (a) we have the inequality (\ref{FirstTypeComponent}), which yields
\[ n \ \le \ \frac{(k - 1)(k + 2) + 2 \cdot \dim \left( T_{L_W} B^k_{1,d} \right) }{4d}. \]
By Martens' theorem as above, we obtain
\[ n \ \le \ \frac{(k - 1)(k + 2) - 4(k-1) + 2d}{4d} \ = \ \frac{(k-1)(k-2)}{4d} + \frac{1}{2} . \]
The above, by Clifford's theorem, becomes
\[ n \ \le \ \frac{\frac{d}{2} \cdot \left(\frac{d-2}{2} \right)}{4d} + \frac{1}{2} \ = \ \frac{d(d-2)}{16d} + \frac{1}{2} . \]
As $d \ne 0$, this simplifies to $n \le \frac{d - 2}{16} + \frac{1}{2}$. As $W$ is stable, $d \le g-2$, whence
\[ n \ \le \ \frac{g - 4}{16} + \frac{1}{2} \ = \ \frac{g + 4}{16} . \qedhere \]
\end{proof}

\subsection{Symplectic bundles of second type}

In \cite[Theorem 4]{Baj19}, the first author derived a bound on the dimension of the Brill--Noether locus $B^k_{2, K}$ of bundles of rank two and canonical determinant. As noted above, these are precisely the $K$-valued symplectic bundles of rank two. The following is a generalisation to symplectic bundles of higher rank, whose proof is similar.\\
\\
\textbf{Notation.} For the remainder of the paper, as we shall only consider symmetric Petri maps, we denote $\mus$ simply by $\mu$ to ease notation.

\begin{theorem} \label{SecondTypeBound} Let $k$ be an integer satisfying $1 \le k \le n(g+1) - 1$. Suppose $Y$ is an irreducible component of $\Snk$ containing a bundle $W$ of second type satisfying $h^0 ( W ) = k$ and such that the rank of the subbundle $E \subset W$ generated by global sections is $r$. Then
\begin{multline*} \dim (Y) \ \le \ \dim ( T_W Y ) \ \le \ \min \left\{ n(2n+1)(g-1) - ( 2k - 1 ) , \right. \\ \left. n(2n+1)(g-1) - k - \frac{1}{2} r(r-1) \right\} . \end{multline*}
\end{theorem}

\begin{proof} Let $W$ be a general element of $Y$. If $\mu \colon \Sym^2 H^0 ( W ) \to H^0 ( \Sym^2 W )$ is the Petri map of $W$, then
\begin{equation} \dim ( T_W Y ) \ = \ \dim ( \MS ) - \frac{1}{2}k(k+1) + \dim \Ker ( \mu ) . \label{BasicMuIneq} \end{equation}
We shall prove the theorem by finding a bound on $\dim \Ker ( \mu )$.

As $W$ is of second type, we may fix an exact sequence $0 \to \Oc (D) \to W \xrightarrow{q} F \to 0$, where $D$ is an effective divisor with $h^0 ( \Oc (D) ) = 1$. Now we have an exact commutative diagram
\[ \xymatrix{ 0 \ar[r] & \Sym^2 H^0 ( \Oc (D) ) \ar[r]^i \ar[d]_{\mu_1} & \Sym^2 H^0 ( W ) \ar[r]^j \ar[d]_\mu & \frac{\Sym^2 H^0 ( W )}{\Sym^2 H^0 ( \Oc (D) )} \ar[d]_{\mu_2} \ar[r] & 0 \\
 0 \ar[r] & H^0 ( \Oc ( 2D ) ) \ar[r] & H^0 ( \Sym^2 W ) \ar[r] & H^0 \left( \frac{\Sym^2 W}{\Oc ( 2D )} \right) . } \]
As $h^0 ( \Oc (D) ) = 1$, clearly $\mu_1$ is injective. 
Thus, by the Snake Lemma,
\begin{equation} \dim \Ker ( \mu ) \ \le \ \dim \Ker ( \mu_2 ) . \label{FirstMuIneq} \end{equation}

Next, write $V$ for the image of $q \colon H^0 ( W ) \to H^0 ( F )$. There is a commutative diagram with exact rows
\[ \xymatrix{ 0 \ar[r] & V \otimes H^0 ( \Oc (D) ) \ar[r]^\iota \ar[d]^\gamma & \frac{\Sym^2 H^0 ( W )}{\Sym^2 H^0 ( \Oc (D) )} \ar[r] \ar[d]_{\mu_2} & \Sym^2 V \ar[d]_{\mu_3} \ar[r] & 0 \\
0 \ar[r] & H^0 ( F(D) ) \ar[r] & H^0 \left( \frac{\Sym^2 W}{\Oc (2D)} \right) \ar[r] & H^0 ( \Sym^2 F ) & } \]
Here $\gamma$ is the multiplication map on sections, and $\iota$ is induced by $\sym \colon H^0 ( W ) \otimes H^0 ( \Oc (D) ) \to \Sym^2 H^0 ( W )$.
As $D$ is effective and $h^0 ( \Oc (D) ) = 1$, the map $\gamma$ is injective. Hence by the Snake Lemma and (\ref{FirstMuIneq}) we have
\begin{equation} \dim \Ker ( \mu ) \ \le \ \dim \Ker ( \mu_3 ) . \label{SecondMuIneq} \end{equation}
Therefore by Lemma \ref{SymPetriBound} below, $\dim \Ker ( \mu )$ is bounded above by
\begin{multline*} \min \left\{ \frac{1}{2}k(k-1) - \frac{1}{2} r(r-1) , \frac{1}{2}k(k-1) - (k-1) \right\} \ = \\
\min \left\{ \frac{1}{2}k(k+1) - \left( k + \frac{1}{2} r(r-1) \right) , \frac{1}{2}k(k+1) - (2k-1) \right\} \end{multline*}
The theorem now follows from (\ref{BasicMuIneq}). \end{proof}

\begin{lemma} \label{SymPetriBound} Let $F$ be any vector bundle, and $V$ a nonzero subspace of $H^0 ( F )$. Let $E$ be the subbundle of $F$ generated by $V$, and write $m := \rank(E)$. Let $\mu_3 \colon \Sym^2 V \to H^0 ( \Sym^2 F )$ be the restriction of the symmetric Petri map of $F$. Then
\[ \dim \im ( \mu_3 ) \ \ge \ \max \left\{ \frac{1}{2} m(m+1) , \dim (V) \right\} . \] \end{lemma}

\begin{proof} Let $\Lambda \subset V$ be a subspace of dimension $m$ which generically generates $E$. Then for generic $p \in C$, the composed map
\[ \Sym^2 \Lambda \ \xrightarrow{\mu_V} H^0 ( \Sym^2 F ) \ \xrightarrow{\ev} \ \Sym^2 F|_p \]
is an isomorphism. Thus $\dim \im ( \mu_3 ) \ge \rank ( \Sym^2 E ) = \frac{1}{2} m (m+1)$.

For the rest: Choose any nonzero $t \in V$, and write $L$ for the line subbundle generated by $t$. There is a commutative diagram
\[ \xymatrix{ \K \cdot t \otimes V \ar@^{(->}[r] \ar@_{(->}[dr] & H^0 ( L ) \otimes V \ar[r]^\sym \ar[d] & \Sym^2 V \ar[d]_{\mu_3} \\
 & H^0 ( L \otimes F ) \ar@^{(->}[r]^\Sigma & H^0 ( \Sym^2 F ) , } \]
where $\Sigma \colon F \otimes F \to \Sym^2 F$ is the canonical surjection. Since $\dim ( \K \cdot t ) = 1$, the top row is injective. On the other hand, since $L$ has rank one, $(L \otimes F) \cap \wedge^2 F = 0$. Thus $\Sigma$ is induced by an injective bundle map, and so is injective. By commutativity, the restriction of $\mu_3$ to $\sym ( \K \cdot t \otimes V )$ is injective. Thus $\dim \im ( \mu_3 ) \ge \dim (V)$. \end{proof}

\begin{remark} We mention some special cases. If $h^0 ( W ) = 1$, then $W$ is both of first and of second type, and Theorems \ref{FirstTypeBound} and \ref{SecondTypeBound} both confirm that $\Sn^1$ is a generically smooth reduced divisor. More generally, if $k = r$, then $W$ belongs to a unique component of $\Snk$ which is generically smooth and of the expected dimension. \end{remark}

\section{Nonemptiness of symplectic Brill--Noether loci} \label{Nonemptiness}

In this section, we shall prove nonemptiness of $\Snk$ for certain values of $g$, $n$ and $k$. We use a combination of techniques from \cite{Mer} and \cite{CH3}. In {\S\S} \ref{MercatBundles} and \ref{ExtGeom} we recall or prove the necessary ingredients, and then proceed to the questions of nonemptiness and smoothness of $\Snk$.

\subsection{Mercat's construction} \label{MercatBundles}

Here we recall and further analyse the bundles constructed on \cite[p.\ 76]{Mer} as elementary transformations of sums of line bundles. Let $C$ be any curve of genus $g \ge 3$. As in the introduction, set
\[ k_0 \ := \ \max\{ k \ge 0 : \beta^k_{1, g-1} > 0 \} . \]
Fix $n \ge 1$. By definition of $k_0$, the Brill--Noether locus $B^{k_0}_{1, g-1}$ is of positive dimension. Let $L_1 , \ldots , L_n$ be general elements of $B^{k_0}_{1, g-1}$, in particular such that
\[ L_1 , \ldots , L_n , K L_1^{-1} , \ldots , K L_n^{-1} \]
are mutually nonisomorphic. Choose any point $x \in C$. Let $E$ be an elementary transformation
\begin{equation} 0 \ \to \ E \ \to \ \bigoplus_{i=1}^n L_i \ \to \ \cO_x \ \to \ 0 \label{defnE} \end{equation}
which is general in the sense that no $L_i|_x$ is contained in $E$. One checks using \cite[p.\ 79]{Mer} that such an $E$ is stable.
 Hence $K \otimes E^*$ is also stable, so any proper subbundle has slope at most $g-1$. In fact we shall require the following stronger statement.

\begin{lemma} \label{FinitelyManyLineSubbs} Suppose $n \ge 2$. Let $E$ be as in (\ref{defnE}).
\begin{enumerate}
\item[(a)] Any slope $g-1$ subbundle of $K \otimes E^*$ contains a line subbundle of degree $g-1$.
\item[(b)] The bundle $K \otimes E^*$ contains a finite number of line subbundles of degree $g-1$.
\end{enumerate}
\end{lemma}

\begin{proof} We use induction on $n$. Firstly, suppose $n = 2$. Recall that the Segre invariant $s_1 ( K \otimes E^* )$ is defined as
\[ \min\{ \deg ( K \otimes E^* ) - 2 \deg (M) : M \hbox{ a line subbundle of } K \otimes E^* \} . \]
As $K L_i^{-1}$ is clearly a maximal line subbundle of $K \otimes E^*$, we have $s_1 ( K \otimes E^* ) 
 = 1$. Then parts (a) and (b) both follow from \cite[Proposition 4.2]{LNar}.

Now suppose $n \ge 3$. We have a diagram
\[ \xymatrix{ 0 \ar[r] & E_n \ar[r] \ar[d] & E \ar[r] \ar[d] & L_n \ar[r] \ar[d]^= & 0 \\
 0 \ar[r] & \bigoplus_{i=1}^{n-1} L_i \ar[r] \ar[d] & \bigoplus_{i=1}^n L_i \ar[r] \ar[d] & L_n \ar[r] & 0 \\
 & \cO_x \ar[r]^= & \cO_x & & } \]
where $E_n$ has rank $n-1$ and degree $(n-1)(g-1) - 1$. Since no $L_i$ is contained in $E$, in particular no $L_i$ is contained in $E_n$. Thus, by induction we may assume that statements (a) and (b) hold for $K \otimes E_n^*$.

We now prove part (a). Suppose $F$ is a slope $g-1$ subbundle of $K \otimes E^*$. We have a diagram of sheaves
\begin{equation} \xymatrix{ 0 \ar[r] & K L_n^{-1} \ar[r] & K \otimes E^* \ar[r] & K \otimes E_n^* \ar[r] & 0 \\
 0 \ar[r] & F_1 \ar[r] \ar[u] & F \ar[r] \ar[u] & F_2 \ar[r] \ar[u] & 0 } \label{IndDiag} \end{equation}
where $F_1$ is the sheaf-theoretic intersection of $F$ and $K L_n^{-1}$. If $F_1 \ne 0$ then $F_1 = K L_n^{-1}$ and we are done. If $F_1 = 0$ then $F \cong F_2$ is a slope $g-1$ subsheaf of $K \otimes E_n^*$. Since the latter is stable of slope $g-1 + \frac{1}{\rank(E_n)}$, in fact $F_2$ must be saturated; that is, a subbundle. By induction, $F \cong F_2$ contains a line subbundle of degree $g - 1$. This proves (a).

As for (b): By the top row of (\ref{IndDiag}), any degree $g-1$ line subbundle $M \subset K \otimes E^*$ is either $K L_n^{-1}$ or is a subbundle of $K \otimes E_n^*$. By induction, we may assume there are at most finitely many degree $g-1$ subbundles of $K \otimes E_n^*$. For a fixed such subbundle $M$, the set of liftings of $M$ to $K \otimes E^*$ is a pseudotorsor over $H^0 ( \Hom ( M, K L_n^{-1} ))$. Since the $L_i$ are chosen generally from the positive dimensional locus $B^{k_0}_{1, g-1}$, perturbing $L_n$ if necessary we can assume that $K L_n^{-1} \not\cong M$, so $h^0 ( \Hom ( M, K L_n^{-1} )) = 0$. Statement (b) follows. \end{proof}

\subsection{Symplectic extensions} \label{ExtGeom}

In this subsection we shall recall a method for constructing symplectic bundles as extensions, together with a geometric criterion for liftings in such extensions.

\begin{criterion} \label{SympExtCriterion} Let $C$ be a curve, and let $E$ be a simple vector bundle over $C$. An extension
\begin{equation} 0 \ \to \ E \ \to \ W \ \to \ K \otimes E^* \ \to \ 0 \label{SympExt} \end{equation}
admits a $K$-valued symplectic form with respect to which $E$ is isotropic if and only if the extension class $\delta (W)$ belongs to $H^1 ( C, K^{-1} \otimes \Sym^2 E )$. \end{criterion}

\begin{proof} This is a special case of \cite[Criterion 2.1]{Hit1}. \end{proof}

Let us now recall some geometric objects living naturally in the projectivised extension space $\PP H^1 ( K^{-1} \otimes E \otimes E )$. Let $V$ be any vector bundle over $C$ with $h^1 ( V ) \ne 0$. Write $\pi \colon \PP V \to C$ for the projection. Via Serre duality and the projection formula, there is a canonical identification
\[ \PP H^1 ( V ) \ \isom \ | \cO_{\PP V} (1) \otimes \pi^* K |^* . \]
Hence there is a natural map $\psi \colon \PP V \dashrightarrow \PP H^1 ( V )$ with nondegenerate image. Let us recall a useful way to realise this map fibrewise.

\begin{lemma} \label{AlternativePsi} On a fibre $\PP V|_y$, the map $\psi$ can be identified with the projectivised coboundary map of the sequence
\[ H^0 ( V ) \ \to \ H^0 ( C, V (y) ) \ \to \ V(y)|_y \ \to \ H^1 ( V ) \to \cdots . \]
\end{lemma}

\begin{proof} This follows by direct calculation, or from the discussion on \cite[pp.\ 469--470]{CH1}. \end{proof}

Now set $V = K^{-1} \otimes E \otimes E$. We shall recall a result from \cite{CH1} relating the geometry of $\psi ( \PP ( E \otimes E ) )$ and liftings of subsheaves of $K \otimes E^*$ to extensions of the form (\ref{SympExt}), in the spirit of \cite[Proposition 1.1]{LNar}. Let $e_1, \ldots , e_m$ be points of $E$ lying over distinct points $y_1, \ldots , y_m \in C$. These define an elementary transformation
\[ 0 \ \to \ F_{e_1 , \ldots , e_m} \ \to \ K \otimes E^* \ \to \ \bigoplus_{l = 1}^m \cO_{y_l} \ \to \ 0 . \]

\begin{proposition} \label{LiftingCriterion} With $E$ and $F := F_{e_1 , \ldots , e_m}$ as above, let $0 \to E \to W \to K \otimes E^* \to 0$ be an extension of class $\delta (W) \in \PP H^1 ( K^{-1} \otimes E \otimes E )$. Then $F$ lifts to $W$ if and only if $\delta (W)$ belongs to the secant spanned by $\psi ( e_1 \otimes f_1 ) , \ldots , \psi ( e_m \otimes f_m )$ for some nonzero $f_1 \in E|_{y_1} , \ldots , f_m \in E|_{y_m}$. \end{proposition}

\begin{proof} Let $\beta \colon H^1 ( K^{-1} \otimes E \otimes E ) \to H^1 ( F^* \otimes E )$ be the induced map on cohomology. Then $F$ lifts to an extension $W$ if and only if $\delta(W) \in \Ker ( \beta )$. By \cite[Lemma 4.3 (ii)]{CH1}, the space $\Ker ( \beta )$ is exactly the span of the projective linear spaces $\psi \left( \PP ( \K \cdot e_l \otimes K^{-1} \otimes E ) \right)$ for $1 \le l \le m$. (Note that the assumption on the degrees in \cite{CH1} is made solely to ensure that $\psi$ be an embedding, which we do not require in the present situation.) \end{proof}

Next, as in \cite[{\S} 2.2]{CH3}, composing $\psi$ with the relative Segre embedding, we obtain a map
\begin{equation} \psis \colon \PP E \ \hookrightarrow \ \PP ( \Sym^2 E ) \ \dashrightarrow \ \PP H^1 ( K^{-1} \otimes \Sym^2 E ) \label{defnPsis} \end{equation}
with nondegenerate image. Note that $\psis (e) = \psi ( e \otimes e )$. We remark that $\psis$ is the map associated to
\[ |\cO_{\PP E} ( 2 ) \otimes \pi^* K^2 |^* \ \cong \ \PP H^0 ( K^2 \otimes \Sym^2 E^* )^* \ \cong \ \PP H^1 ( K^{-1} \otimes \Sym^2 E ) . \]

\subsection{The construction} \label{Construction}

Suppose $g \ge 3$ and $n \ge 1$. Let $L_1 , \ldots , L_n$ and $E$ be as defined in {\S} \ref{MercatBundles}. Let $e_1, e_2$ be general points of $\PP E$ lying over distinct $y_1, y_2 \in C$ respectively. Let
\begin{equation} 0 \ \to \ E \ \to \ W \ \to \ K \otimes E^* \ \to \ 0 \label{defnW} \end{equation}
be a nontrivial extension such that $\delta (W)$ is a general point of the line spanned by $\psis (e_1)$ and $\psis (e_2)$. As $\delta( W ) \in H^1 ( K^{-1} \otimes \Sym^2 E )$, by Criterion \ref{SympExtCriterion} there is a $K$-valued symplectic structure on $W$.

\begin{proposition} The bundle $W$ is stable as a vector bundle. \label{WStable} \end{proposition}

\begin{proof} The following uses ideas from \cite[{\S} 3]{CH3} and \cite[Lemma 7]{HP}. As every proper subbundle of $K \otimes E^*$ has slope at most $g-1$, and the extension $W$ is nontrivial, it is not hard to see that any subbundle of $W$ has slope at most $g-1$. Thus we need only to exclude the existence of a subbundle of slope $g-1$.

Furthermore, for any proper subbundle $F \subset W_1$, we have a short exact sequence $0 \to F^\perp \to W \to F^* \otimes K \to 0$ where $F^\perp$ is the orthogonal complement of $F$ with respect to the bilinear form. An easy computation shows that
\[ \mu ( F^\perp ) \ = \ (g-1) + \frac{\rank (F)}{2n - \rank(F)} \left( \mu (F) - (g-1) \right) . \]
Hence $\mu (F) \ge (g-1)$ if and only if $\mu ( F^\perp ) \ge (g-1)$. As $\rank(F^\perp) = 2n - \rank(F)$, to prove stability of $W$ it suffices to exclude the existence of subbundles of slope $g-1$ and rank at most $n$.

Let $F \subset W$ be a subbundle of rank at most $n$. Then there is a sheaf diagram
\[ \xymatrix{ 0 \ar[r] & F_1 \ar[r] \ar[d] & F \ar[r] \ar[d] & F_2 \ar[r] \ar[d] & 0 \\
 0 \ar[r] & E \ar[r] & W \ar[r] & K \otimes E^* \ar[r] & 0, } \]
where $F_1$ is a subbundle of $E$ and $F_2$ a subsheaf of $K \otimes E^*$. For $j = 1, 2$ write $r_j := \rank (F_j)$. 
 As $\mu ( F_2 ) < g-1 + \frac{1}{n}$, in fact $\mu (F_2) \le g-1$. Therefore, if $r_1 \ne 0$ then
\[ \mu (F) \ \le \ \frac{ r_1 \cdot \mu (E) + r_2 \cdot (g-1)}{r_1 + r_2} \ < \ g-1 . \]
Thus we may assume that $r_1 = 0$ and $F \cong F_2$ is a subsheaf of $K \otimes E^*$.

If $r_2 < n$, then by Lemma \ref{FinitelyManyLineSubbs} (a) we may assume $n \ge 2$ and $r_2 = 1$. Let $\iota \colon M \to K \otimes E^*$ be a line subbundle of degree $g-1$. Then $\iota$ lifts to a map $M \to W$ if and only if
\[ \delta (W) \ \in \ \Ker \left( \iota^* \colon H^1 ( C, K^{-1} \otimes E \otimes E ) \ \to \ H^1 ( C, M^{-1} \otimes E ) ) \right) . \]
As $H^1 ( C, M^{-1} \otimes E ) )$ is nonzero, by Lemma \ref{jstar}, the restriction of $\iota^*$ to $H^1 ( K^{-1} \otimes \Sym^2 E )$ is nonzero. Furthermore, by Lemma \ref{FinitelyManyLineSubbs} (b), there are only finitely many possibilities for $\iota$. We conclude that the locus of extensions in $H^1 ( K^{-1} \otimes \Sym^2 E )$ admitting a lifting of some such $\iota \colon M \to K \otimes E^*$ is a finite union of proper linear subspaces. Since
\[ \psis ( \PP E ) \ \subset \ \PP H^1 ( K^{-1} \otimes \Sym^2 E) \ \cong | \cO_{\PP E} (2) \otimes \pi^* K^2 |^* \ \]
is nondegenerate and $\delta(W)$ is a general point of a general 2-secant to $\psis ( \PP E )$, we may assume that $\delta (W)$ does not belong to any of these proper linear subspaces.

Finally, we must exclude a lifting of some $F_2$ of rank $r_2 = n \ge 1$; that is, an elementary transformation $0 \to F_2 \to K \otimes E^* \to \cO_y \to 0$. By Proposition \ref{LiftingCriterion}, such a lifting exists only if $\delta (W)$ belongs to $\psi ( \Delta )$, where
\[ \Delta \ := \PP E \times_C \PP (K^{-1} \otimes E) \]
is the rank one locus of $\PP ( K^{-1} \otimes E \otimes E )$.

 Now $h^0 ( K^{-1} \otimes \Sym^2 E ) = 0$ since $E$ is stable of slope $< g-1$. Hence by Riemann--Roch,
\[ h^1 ( K^{-1} \otimes \Sym^2 E ) \ = \ \frac{1}{2}n(n+1)(g-1) + n + 1 . \]
One checks easily that for $g \ge 3$, this is greater than $\dim ( \PP E ) + 1 = n + 1$, so $\psis ( \PP E )$ is a proper subvariety of $\PP H^1 ( K^{-1} \otimes \Sym^2 E )$.
 It follows that the secant variety $\Sec^2 ( \psis ( \PP E ) )$ strictly contains $\psis ( \PP E )$. Hence, since the points $e_1, e_2$ were chosen generally and $\delta (W)$ is general in the line $\overline{\psis (e_1) \psis (e_2)}$, we may assume $\delta(W) \not\in \psis ( \PP E )$. Thus $\delta(W)$ belongs to $\psi (\Delta)$ only if $\psi ( e \otimes f ) \in \PP H^1 ( K^{-1} \otimes \Sym^2 E )$ for some independent $e, f$ in some fibre $E|_y$. In view of Lemma \ref{AlternativePsi} and the diagram
\[ \xymatrix{ H^0 ( K^{-1} \otimes E \otimes E (y) ) \ar[r] & K^{-1} \otimes E \otimes E (y)|_y \ar[r] & H^1 ( K^{-1} \otimes E \otimes E ) \\
 & & H^0 ( K^{-1} \otimes \Sym^2 E ) \ar[u] } \]
this happens if and only if there is a global section $\alpha$ of $K^{-1} \otimes E \otimes E (y)$ with value $\frac{1}{2} ( e \otimes f - f \otimes e )$ at $y$. 
 We claim that such an $\alpha$ can exist for at most finitely many $y$. Since $K^{-1} \otimes E \otimes E (y)$ is a subsheaf of $\bigoplus_{i, j} K^{-1} L_i L_j (y)$, it suffices to show that for each $i, j$ the degree $1$ line bundle $K^{-1} L_i L_j (y)$ is effective for at most finitely many $y \in C$. This follows from Lemma \ref{technical} below.

Therefore, writing $\Delta'$ for the complement of the relative diagonal $\PP E \subset \Delta$, the intersection of $\psi ( \Delta' )$ with $H^1 ( K^{-1} \otimes \Sym^2 E )$ is contained in at most a finite number of fibres $\Delta'|_y$. As the linear span of $\psi ( \Delta'|_y )$ is $\psi ( \PP ( K^{-1} \otimes E \otimes E )|_y )$, we conclude that the locus of extensions in $H^1 ( K^{-1} \otimes \Sym^2 E )$ lying over $\psi ( \Delta' )$ is contained in a finite union of linear subspaces of dimension at most $n^2$. 
Again, one computes using $g \ge 3$ that $h^1 ( K^{-1} \otimes \Sym^2 E ) > n^2$.
 Thus the locus of symplectic extensions (\ref{SympExt}) admitting a lifting of an elementary transformation $F_2 \subset K \otimes E^*$ with $\deg \left( \frac{K \otimes E^*}{F_2} \right) = 1$ is contained in a finite union of proper linear subspaces. As above, by nondegeneracy of $\psis ( \PP E )$ we can assume that $W$ does not admit such a lifting. This completes the proof that $W$ is stable as a vector bundle. \end{proof}

\begin{lemma} \label{technical} Let $C$ be a curve of genus $g \ge 2$. Let $M$ be a nontrivial line bundle of degree zero. Then $h^0 ( M (x) ) = 0$ for general $x \in C$. \end{lemma}

\begin{proof} If $h^0 ( M (x) ) \ge 1$ for general $x \in C$, then for each $x$ we have $M = \Oc(\alpha(x) - x)$ for some $\alpha \in \Aut(C)$.
 As $M \not\cong \Oc$, in fact $\alpha$ has no fixed points. Then for any $x, y \in C$ we have $M = \Oc(\alpha(x) - x) = \Oc(\alpha(y) - y)$, whence $\Oc ( \alpha(x) + y ) = \Oc ( x + \alpha (y) )$. For $y \ne x$, the divisors $\alpha(x) + y$ and $x + \alpha (y)$ are distinct and linearly equivalent divisors. But it is easy to see that this implies that a general divisor of degree two on $C$ moves in a linear series, which can happen only if $g \le 1$.
\end{proof}

\begin{theorem} \label{SnkNonempty} Let $C$ be a curve of genus $g \ge 3$, and let $k_0$ be as defined in (\ref{Defnko}). For each $n \ge 1$ and for $0 \le k \le 2nk_0 - 3$, the locus $\Snk$ has a component which is nonempty and of codimension at most $\frac{1}{2}k(k+1)$. \end{theorem}

\begin{proof} Let $W$ be the $K$-valued symplectic bundle constructed in (\ref{defnW}), which is stable by Proposition \ref{WStable}. By Proposition \ref{LiftingCriterion}, the elementary transformation
\[ 0 \ \to \ F_{e_1, e_2} \ \to K \otimes E^* \ \xrightarrow{e_1, e_2} \ \cO_{y_1} \oplus \cO_{y_2} \ \to \ 0 \]
lifts to a subsheaf $F$ of $W$ (which is in fact a subbundle, as $W$ is stable). Since $e_1$ and $e_2$ are general and $K \otimes E^*$ is generically generated, we may assume $h^0 ( F ) = nk_0 - 2$. Hence $h^0 ( W ) \ge h^0 ( E ) + h^0 ( F ) = 2nk_0 - 3$ and $W$ defines a point of $\Snk$. In particular, $\Snk$ is nonempty. By Proposition \ref{SchemeStructure} (b), each component is of codimension at most $\frac{1}{2}k(k+1)$. \end{proof}

\begin{remark} \label{Semistable} If one allows strictly semistable symplectic bundles, it is easy to give examples of $K$-valued symplectic bundles with larger $h^0$ over any curve. Set
\[ k_1 \ := \ \max\{ h^0 ( L ) : L \in \Pic^{g-1} (C) \} . \]
Let $L_1 , \ldots , L_n$ be (not necessarily pairwise nonisomorphic) line bundles of degree $g-1$ with $h^0 ( L_i ) \ge k_1$. Then the direct sum
\[ W \ := \ \bigoplus \left( L_i \oplus K L_i^{-1} \right) \]
endowed with the sum of the standard skewsymmetric forms on the $L_i \oplus K L_i^{-1}$
 is semistable (but not stable) $K$-valued symplectic of rank $2n$ with $h^0 ( W ) = 2nk_1 > 2nk_0 - 3$. \end{remark}

\subsection{Smoothness} \label{Smoothness}

Now we shall prove that if $C$ is a general Petri curve, the component of $\Snk$ whose existence was shown above is smooth and of the expected codimension $\frac{1}{2}k(k+1)$. We shall require the following lemma, whose proof is straightforward.

\begin{lemma} \label{PetriInjSum} Let $V$ be a vector bundle. Suppose $F_1 , \ldots , F_m$ are sheaves such that $\bigoplus_{i=1}^m F_i$ is a subsheaf of $V$ with $H^0 ( V ) = \bigoplus_{i=1}^m H^0 ( F_i )$. Suppose that the multiplication maps
\[ H^0 ( F_i ) \otimes H^0 ( F_j ) \ \to \ H^0 ( F_i \otimes F_j ) \quad \hbox{and} \quad \Sym^2 H^0 ( F_i ) \ \to \ H^0 ( \Sym^2 F_i ) \]
are injective for $1 \le i \le j \le m$. Then the Petri map $\Sym^2 H^0 ( V ) \to H^0 ( \Sym^2 V )$ is injective. \end{lemma}

\begin{theorem} \label{GenSmoothExpDim} Let $C$ be a general Petri curve of genus $g \ge 3$. Then for $n \ge 2$ and $k \le 2nk_0 - 3$, the locus $\Snk$ has a component which is generically smooth and of the expected dimension. \end{theorem}

\begin{remark} Note that the Petri assumption implies that $k_0 = \left\lfloor \sqrt{g - 1} \right\rfloor$. 
\end{remark}

\begin{proof}[Proof of Theorem \ref{GenSmoothExpDim}] Recall the $K$-valued symplectic bundle $W$ constructed in (\ref{defnW}), which by Proposition \ref{WStable} defines a point of $\cS_{2n, K}^{2nk_0 - 3}$. By Corollary \ref{SmoothCrit} and Lemma \ref{OneImpliesMore}, the statement will follow if we can show that $\mu \colon \Sym^2 H^0 ( W ) \to H^0 ( \Sym^2 W )$ is injective.

The following argument is modelled upon the proof of \cite[Lemma 7.2]{HHN}. Let $p \in C$ be a point which is not a base point for any $K L_i^{-1}$, so $h^0 ( L_i (p) ) = h^0 ( L_i )$ for $1 \le i \le n$. For each $i$, we have a commutative diagram
\[ \xymatrix{ H^0 ( L_i ) \otimes H^0 ( K L_i^{-1} ) \ar[d]^\wr \ar[r] & H^0 ( K ) \ar[d]^\wr \\
 H^0 ( L_i (p) ) \otimes H^0 ( K L_i^{-1} ) \ar[r] & H^0 ( K (p) ). } \]
Now let $U$ be the open subset of $B^{k_0}_{1, g-1}$ over which $h^0 ( L ) = h^0 ( L (p) ) = k_0$. (Note that since $C$ is Petri, $B^{k_0}_{1, g-1}$ is irreducible by \cite[Remark 4.2]{HHN}.) Let $\cA$ and $\cB$ be vector bundles over $U \times U$ whose fibres at $(L, N)$ are $H^0 ( L (p) ) \otimes H^0 ( K N^{-1} )$ and $H^0 ( K L N^{-1} (p) )$ respectively. These have rank $k_0^2$ and $g$ respectively. Let $\tilde{\mu} \colon \cA \to \cB$ be the globalised Petri map. Since $C$ is Petri, the composed map
\[ H^0 ( L_i ) \otimes H^0 ( K L_i^{-1} ) \ \to \ H^0 ( K ) \ \to \ H^0 ( K(p) ) \]
is injective for all $L_i$. Hence $\tilde{\mu}$ is injective on an open subset of $U \times U$. Deforming the $L_i$ if necessary, we may assume 
 that the multiplication maps
\begin{multline} H^0 ( L_i ) \otimes H^0 ( L_j ) \ \to \ H^0 ( L_i L_j ) \quad \hbox{and} \quad H^0 ( L_i ) \otimes H^0 ( K L_j^{-1} ) \ \to \ H^0 ( K L_i L_j^{-1} ) \\
\hbox{and} \quad H^0 ( K L_i^{-1} ) \otimes H^0 ( K L_j^{-1} ) \ \to H^0 ( K^2 L_i^{-1} L_j^{-1} ) \label{PIone} \end{multline}
are injective for all $i, j$.

Furthermore, as $C$ is now assumed general in moduli and the $L_i$ were chosen generally in the positive-dimensional locus $B^{k_0}_{1, g-1}$, by \cite[Theorem 1]{Bal12} the symmetric Petri maps
\begin{equation} \Sym^2 H^0 ( L_i ) \ \to \ H^0 ( L_i^2 ) \quad \hbox{and} \quad \Sym^2 H^0 ( K L_i^{-1} ) \ \to \ H^0 ( K^2 L_i^{-2} ) \label{PItwo} \end{equation}
are injective for all $i$.

Next, from the proof of Proposition \ref{WStable} we recall the subbundle $F \subset W$ lifting from the elementary transformation $F_{e_1 , e_2} \subset K \otimes E^*$. We claim that $H^0 ( W ) = H^0 ( E ) \oplus H^0 ( F )$. One direction is clear. For the rest; note that there is a commutative diagram
\[ \xymatrix{ H^0 ( K \otimes E^* ) \otimes \left(  K^{-1} \otimes E \otimes E(y_1 + y_2)|_{y_1 + y_2} \right) \ar[r]^-\varepsilon \ar[d] & E(y_1 + y_2)|_{y_1 + y_2} \ar[d] \\
H^0 ( K \otimes E^* ) \otimes H^1 ( K^{-1} \otimes E \otimes E ) \ar[r]^-\cup & H^1 ( E ) } \]
where the vertical arrows are induced by coboundary maps. Perturbing $e_1$ and $e_2$ if necessary, we may assume that $K \otimes E^*$ is generated at $y_1$ and $y_2$, and thus that $\dim \im ( \varepsilon ) = 2$. Then by commutativity and in view of Lemma \ref{AlternativePsi} (with $V = E$), the projectivised image of $\cup \delta (W)$ is spanned by the images of $e_1$ and $e_2$ in $\PP H^1 ( E ) = | \cO_{\PP E} (1) \otimes \pi^* K |^*$. Perturbing again if necessary, we may assume that these images span a $\PP^1$. We conclude that $\cup \delta(W)$ has rank 2, whence $h^0 ( W ) = 2k_0 - 3$ and $H^0 ( W ) = H^0 ( E ) \oplus H^0 ( F )$ as desired. As
\[ H^0 ( E ) \ \subset \ \bigoplus_i H^0 ( L_i ) \quad \hbox{and} \quad H^0 ( F ) \ \subset \ \bigoplus_j H^0 ( K L_j^{-1} ) , \]
by injectivity of the maps in (\ref{PIone}) and (\ref{PItwo}) and by Lemma \ref{PetriInjSum}, we obtain the injectivity of $\mu \colon \Sym^2 H^0 ( W ) \to H^0 ( \Sym^2 W )$. This completes the proof. \end{proof}

\begin{remark} Recall that the scheme $\Snk$ has expected dimension 
\[\beta^k_{2n, s}(K):=n(2n+1)(g-1)-\frac{1}{2}k(k+1).\]
In the case $2n=2$, Bertram and Feinberg conjectured in \cite{BF}, that if the expected number 
\[\beta^k_{2, s}(K)= 3g-3-\frac{1}{2}k(k+1)\]
is non-negative, then $\mathcal{S}^k_{2, K}=B^k(2, K)$ would be non-empty. They further predicted that on a general curve, $\mathcal{S}^k_{2, K}$ would be non-empty only if $\beta^k_{2, s}(K)\geq 0$. Mukai states this conjecture as a problem in \cite[Problem 4.11]{Muk92} and \cite[Problem 4.8]{Muk97}.

Montserrate Teixidor proves in \cite[Th. 1.1]{TiB07} that on a general curve if If $k = 2k_1$, then $\mathcal{S}^k_{2, K}$ is non-empty when $g \geq k^2_1$ if $k_1 > 2, g \geq 5$ if $k_1 = 2, g \geq 3$ if $k_1 = 1$. Moreover, under these conditions, it has a component of the right dimension $\rho^k_K$. As for the case $k = 2k_1+1$ she proves that $\mathcal{S}^k_{2, K}$ is non-empty when $g \geq (k_1)^2+k_1+1$.
Moreover, under these conditions, it has a component of the right dimension. 

H. Lange, P. E. Newstead and Seong Suk Park proved in \cite{LNS} that if $C$ is a general curve of odd prime genus $g$ and if $g-1\geq \max\{2k-1, \frac{k(k-1)}{4}\}$, then $\mathcal{S}^k_{2, K}$ would be non-empty.

Our results, Theorems \ref{SnkNonempty} and \ref{GenSmoothExpDim} push forward the Bertram--Feinberg--Mukai conjecture and encounter it when $n\geq 2$. Our Theorem \ref{SnkNonempty} doesn't need any generic condition. Furthermore, while the existing mentioned results carry no results for smoothness of $\mathcal{S}^k_{2, K}$, Theorem \ref{GenSmoothExpDim} verifies the generic smoothness of an appropriately constructed component of $\Snk$. An optimized bound for $k$ in theorems \ref{SnkNonempty} and \ref{GenSmoothExpDim} needs independent studies. 
\end{remark}
\section{Superabundant components of Brill--Noether loci} \label{Superabundance}

The usual Brill--Noether locus $B^k_{r, d}$ has expected dimension
\[ \beta^k_{r, d} \ = \ \dim \cU (r, d) ) - k ( k - d + r(g-1)) . \]
 As outlined in the introduction, examples of components of excess dimension are relevant both to Brill--Noether theory and the study of determinantal varieties. Building on the observation \cite[{\S} 9]{New11} that $B^k_{2, K}$ can have larger expected dimension than the locus $B^k_{2, 2g-2}$ containing it, we shall now show for infinitely many $n$ and $g$ the existence of superabundant components of $B^k_{2n, 2n(g-1)}$ for a general curve of genus $g$.

The expected dimension of $\Snk$ exceeds that of $B^k_{2n, 2n(g-1)}$ if and only if
\[ \dim \MS - \frac{1}{2}k(k+1) \ > \ \dim \Utn - k ( k - d + r(g-1)) , \]
which is equivalent to
\begin{equation} \label{kCriticalValue} \frac{1}{2}k(k-1) \ > \ n(2n - 1) ( g-1 ) + 1 . \end{equation}

Thus if $\Snk$ is nonempty for a value of $k$ satisfying this inequality, there exists a superabundant component of $B^k_{2n, 2n(g-1)}$. We shall give examples using Theorem \ref{SnkNonempty}. Firstly, for certain values of $g$, one can obtain statements for all $n$.

\begin{theorem} \label{OversizedBNlocus1} Suppose $m \ge 7$ and let $C$ be a general curve of genus $g = m^2 + 1$. Then for any $n \ge 1$, the locus $\Sn^{2nm - 3}$ is nonempty and has dimension greater than $\beta^{2nm-3}_{2n, 2n(g-1)}$. In particular $B^{2nm-3}_{2n, 2n(g-1)}$ has a superabundant component. \end{theorem}

\begin{proof} As $C$ is general, $k_0 = \left\lfloor \sqrt{g-1} \right\rfloor = m$ and the bundle $W$ defined in (\ref{defnW}) defines a point of $\Sn^{2nm - 3}$. For $k = 2nm - 3$, the inequality (\ref{kCriticalValue}) becomes
\[ \frac{(2 n m - 3)(2n m - 4)}{2} \ > \ n(2n-1)m^2 + 1 . \]
The $n^2$-terms cancel,
and the inequality reduces to $n m^2 - 7 n m + 5 \ > \ 0$. One checks easily that this holds for all $n \ge 1$ when $m \ge 7$.
\end{proof}

With the same approach, if we fix $n$, then we can obtain a statement for a general curve of large enough genus. 

\begin{theorem} \label{OversizedBNlocus2} Fix $n \ge 1$ and let $C$ be any curve of genus $g \ge (4n+7)^2 + 1$. Then $\Sn^{2nk_0 -3}$ is nonempty and has dimension greater than $\beta^k_{2n, 2n(g-1)}$. In particular, for fixed $n \ge 1$, there are infinitely many $g$ such that $B^k_{2n, 2n(g-1)}$ has a superabundant component for any curve $C$ of genus $g$. \end{theorem}

\begin{proof} Let $W$ be as above. As $k_0 = \left\lfloor \sqrt{g-1} \right\rfloor$, we have $k_0^2 \le g-1$ but $(k_0 + 1)^2 \ge g$, so 
\begin{equation} \sqrt{g} - 1 \ \le \ k_0 \ \le \ \sqrt{g-1} . \label{kzeroIneq} \end{equation}
Now let us check inequality (\ref{kCriticalValue}) for $k = h^0 ( W) = 2nk_0 - 3$; explicitly, that
\[ \frac{(2nk_0 - 3)(2nk_0 - 4)}{2} \ > \ n(2n-1)(g-1) + 1 , \]
that is,
\begin{equation} 2 n^2 k_0^2 - 7 n k_0 + 5 \ > \ 2n^2 (g-1) - n(g-1) . \label{kCriticalValueAgain} \end{equation}
Rewriting the left side as $2 n^2 ( k_0^2 + 2k_0 ) - 4n^2 k_0 - 7 n k_0 + 6$ and noting that $k_0^2 + 2k_0 \ge g-1$ by the left hand inequality in (\ref{kzeroIneq}), we see that (\ref{kCriticalValueAgain}) would follow from the inequality
\[ - 4n^2 k_0 - 7 n k_0 + 5 \ > \ - n ( g-1 ) , \]
that is, $(g-1) + \frac{5}{n} \ > k_0 ( 4n + 7 )$. As $k_0 \le \sqrt{g-1}$ by (\ref{kzeroIneq}), this would follow from
\[ \sqrt{g-1} \left( 1 + \frac{5}{n ( g-1 )} \right) \ > \ 4n + 7 . \]
This follows from the hypothesis $g \ge (4n+7)^2 + 1$. \end{proof}

Setting $n = 1$, the above theorem shows in particular:

\begin{corollary} For any curve of genus $g \ge 122$, there exist Brill--Noether loci with superabundant components. \end{corollary}

\begin{remark} The bundle $W$ is not a smooth point of the component of $B^{2nk_0-3}_{2n, 2n(g-1)}$. The usual Petri map is identified with the multiplication $H^0 ( W ) \otimes H^0 ( W ) \to H^0 ( W \otimes W )$. Since $W$ has at least one line subbundle $L_1$ with at least two independent sections, the restriction of this map to $\wedge^2 H^0 ( W )$ has nonzero kernel containing $\wedge^2 H^0 ( L_1 )$. Note moreover that we have only shown that $\Sn^{2nk_0 - 3}$ is contained in a superabundant component of $B^{2nk_0 - 3}_{2n, 2n(g-1)}$; this component could in general have even larger dimension. \end{remark}

\begin{remark} In \cite{CFK18}, the authors show that in rank two for a general $\nu$-gonal curve, the superabundant components of $B^k_{2, d}$ are all of first type (cf.\ Definition \ref{TwoTypes}). However, $W$ is generically generated, since $E$ is generically generated and the subspace $H^0 ( F )$ lifting from $H^0 ( K \otimes E^* )$ generically generates $F$. 
This is another aspect in which the higher rank case differs from the rank two case. \end{remark}

\subsection{Superabundant components of moduli of coherent systems} \label{SuperCohSys}

Coherent systems on $C$ were briefly mentioned in {\S} \ref{sectionDeSing}. We recall now some more facts, referring the reader to \cite{Brad} for more information and references; and to \cite{BGPMN} for the connection to Brill--Noether theory. For a coherent system $(W, \Lambda)$ of type $(r, d, k)$ on $C$ and a real number $\alpha$, recall that the \textsl{$\alpha$-slope} of $(W, \Lambda)$ is defined to be the real number
$$\mu_\alpha(W,\Lambda) := \frac{d}{r}+\alpha\frac{k}{r}.$$
The coherent system $(W, \Lambda)$ is called \textsl{$\alpha$-stable} 
 if for any coherent subsystem $(V, \Pi)$ of $(W, \Lambda)$ one has $\mu_\alpha(V, \Pi) < \mu_\alpha(W, \Lambda)$.  
 For any real number $\alpha > 0$, there exists a moduli space $G(\alpha; r, d, k)$ parametrising $\alpha$-stable coherent systems, which has expected dimension
\[ \beta^k_{r, d} \ = \ r^2 ( g-1 ) + 1 - k ( k - d + r (g-1) ) . \]
Furthermore there is an increasing finite sequence of real numbers $0 = \alpha_0 , \alpha_1 , \alpha_2 , \ldots , \alpha_L$ with the property that if $\alpha$ and $\alpha'$ belong to the open interval $(\alpha_i , \alpha_{i+1})$ then $G(\alpha; r, d, k) \cong G(\alpha'; r, d, k)$. The numbers $\alpha_i$ are called \textsl{critical values} for the type $(r, d, k)$.

For any $L \in \Pic^d (C)$, we may also consider the closed sublocus
\[ \GaL \ := \ \{ ( W, \Lambda ) \in G ( \alpha ; r, d, k ) : \det W \cong L \} . \]
It is clear that every component of $\GaL$ has dimension at least $\beta^k_{r, d} - g$. However, in \cite{GN}, the authors show that in several cases this is not sharp, and conjecture in \cite[{\S} 2]{GN} that every component of $G( \alpha; r, L , k )$ has dimension at least
\begin{equation} \beta^k_{r, d} - g + \binom{k}{2} \cdot h^1 ( L ) \ =: \ \gamma^k_{r, L} . \label{ExpDimGaL} \end{equation}

\noindent We have the following result on superabundant components of moduli of coherent systems.

\begin{theorem} \label{OversizedCoherentSystems} Let $C$ be a general curve of genus $g \ge 3$. As before, write $k_0 = \left\lfloor \sqrt{g-1} \right\rfloor$, and $W$ be the $K$-valued symplectic bundle constructed in (\ref{defnW}). Set $k = 2nk_0 - 3$. Let $\alpha_1$ be the smallest positive critical value for the type $(2n, 2n(g-1), k)$, and suppose $0 < \alpha < \alpha_1$.
\begin{enumerate}
\item[(a)] The coherent system $( W, H^0 ( W ) )$ is $\alpha$-stable.
\item[(b)] The fixed determinant locus $G (\alpha; 2n, K^n , k )$, and hence also the full moduli space $G ( \alpha; 2n, 2n(g-1), k )$, contains a component of dimension at least
\[ n(2n+1)(g-1) - \frac{1}{2}k(k+1) . \]
\item[(c)] Suppose $m \ge 7$ and $g = m^2 + 1$, so $k = 2nm - 3$. Then for any $n \ge 1$, the component of $G(\alpha; 2n, 2n(g-1) , 2nm - 3)$ referred to in (b) is superabundant. Moreover, $G ( \alpha; 2n, K^n, 2nm - 3)$ has a component of dimension larger than $\gamma^{2nm-3}_{2n, K^n} + g$ (cf. (\ref{ExpDimGaL})).
\item[(d)] Fix $n \ge 1$ and $g \ge (4n+7)^2 + 1$. Then the component of $G(\alpha; 2n, 2n(g-1) , 2nk_0 - \nolinebreak 3)$ referred to in (b) is superabundant. Moreover, $G ( \alpha; 2n, K^n, 2nk_0 - 3)$ has a component of dimension larger than $\gamma^{2nk_0 - 3}_{2n, K^n} + g$.
\end{enumerate} \end{theorem}

\begin{proof} (a) (See also \cite{KN}.) By Proposition \ref{WStable}, the bundle $W$ is stable. In particular, if $V$ is a proper subbundle of rank $r$, then $\mu ( V ) \le \mu (W) - \frac{1}{2nr}$. It is then easy to check that the coherent system $(W, H^0 ( W ) )$ is $\alpha$-stable for $0 < \alpha < \frac{1}{2nk}$.
 Since $G(\alpha; r, d, k) \cong G(\alpha'; r, d, k)$ for any $\alpha, \alpha'$ in the interval $(0 , \alpha_1 )$, the coherent system $( W, H^0 (W) )$ is $\alpha$-stable for $0 < \alpha < \alpha_1$.

(b) Denote by $X$ the component of $\Snk$ containing $W$. By part (a), for generic $W' \in X$ the coherent system $(W', H^0 (W') )$ is $\alpha$-stable, so there is a map
\[ X \ \dashrightarrow \ G ( \alpha ; 2n, 2n(g-1) , 2nk_0 - 3 ) \]
given by $W' \mapsto ( W', H^0 ( W' ))$. Clearly this is generically injective. In particular, the moduli space $G ( \alpha; 2n, 2n(g-1), 2nk_0 - 3 )$ has a component of dimension at least $n(2n+1)(g-1) - \frac{1}{2}k(k+1)$. Moreover, as any $K$-valued symplectic bundle has determinant $K^n$, the image of $X$ is contained in the fixed determinant locus $\GaK$.

Finally, as $G ( \alpha ; 2n, 2n(g-1), k )$ has the same expected dimension as $B^k_{2n, 2n(g-1)}$, parts (c) and (d) follow from the computations in the proofs of Theorems \ref{OversizedBNlocus1} and \ref{OversizedBNlocus2}. \end{proof}

\end{document}